\documentclass[a4paper, 12pt]{article}
\usepackage{amsmath,amsthm}
\usepackage{amssymb,latexsym}
\usepackage{graphicx,subfig}

\newtheorem{thm}{Theorem}[section]
\newtheorem{dfn}[thm]{Definition}
\newtheorem{lma}[thm]{Lemma}
\newtheorem{cor}[thm]{Corollary}
\newtheorem{prp}[thm]{Proposition}

\newtheorem{clm}[thm]{Claim}

\newtheorem{conj}[thm]{Conjecture}

\begin{document}

\title{Cliques in graphs with bounded minimum degree}
\author{Allan Siu Lun Lo\footnote{DPMMS, University of Cambridge, Cambridge CB3 0WB, UK. Email: allan.lo@cantab.net This author is supported by EPSRC.}}

\maketitle

\abstract{
Let~$k_r(n,\delta)$ be the minimum number of~$r$-cliques in graphs with~$n$ vertices and minimum degree~$\delta$.
We evaluate~$k_r(n,\delta)$ for~$\delta \leq 4n/5$ and some other cases.
Moreover, we give a construction, which we conjecture to give all extremal graphs (subject to certain conditions on~$n$,~$\delta$ and~$r$).
}

\section{Introduction} \label{sec:introduction}

Let $f_r(n,e)$ be the minimum number of $r$-cliques in graphs of order~$n$ and size~$e$.
Determining $f_r(n,e)$ has been a long studied problem.
The case $r=3$, that is counting triangles, has been studied by various people. Erd\H{o}s~\cite{MR0252253}, Lovasz and Simonovits~\cite{LovaszSimonovits83} studied the case when $e = \binom{n}2/2 + l$ with $0 < l \le n/2$.
Fisher~\cite{Fisher89} considered the situation when $\binom{n}2/2 \le e \le 2\binom{n}2/3$, but it was not until nearly twenty years later that a dramatic breakthrough of Razborov~\cite{Razborov08} established the asymptotic value of $f_3(n,e)$ for a general~$e$. 
The proof of this used the concept of flag algebra developed in~\cite{MR2371204}.
Unfortunately, it seemed difficult to generalise Razborov's proof even for~$f_4(n,e)$.
Nikiforov~\cite{Nikiforov08} later gave a simple and elegant proof of the asymptotic values of both $f_3(n,e)$ and $f_4(n,e)$ for general~$e$. 
However, the asymptotic value of $f_r(n,e)$ for $r \ge 5$ have not yet been determined, and the best known lower bounds was given Bollob\'as~\cite{MR0424614}.

In this paper, we are interested in a variant of $f_r(n,e)$, where instead of considering the number of edges we consider the minimum degree.
Define $k_r(n,\delta)$ to be the minimum number of $r$-cliques in graphs of order~$n$ with minimum degree~$\delta$.
In addition, $k_r^{reg}(n,\delta)$ is defined to be the minimum number of $r$-cliques in $\delta$-regular graphs of order~$n$.
It should be noted that there exist $n$ and $\delta$ such that $k_r(n,\delta) =0$, but $k_r^{reg}(n,\delta)>0$.
For example, if $r=3$, $n$ odd and $2n/5 < \delta n<2$, then it is easy to show that $k_3(n,\delta)=0$.
However, a theorem of Andr{\'a}sfai, Erd{\H{o}}s and S{\'o}s~\cite{MR0340075} states that every triangle-free graph of order~$n$ with minimal degree greater than $2n/5$ is bipartite.
Since no regular graph with an odd number of vertices can be bipartite, $k_3^{reg}(n,\delta)>0$ for $n$ odd and $2n/5 < \delta < n/2$, whilst $k_3(n,\delta)=0$.
The author~\cite{Lo08half} evaluated $k_3^{reg}(n,\delta)$ for $n\ge 10^7$ odd and $2n/5 + \sqrt{n}/5 \leq \delta \leq n/2$.

Let $\delta = (1-\beta)n$ with $0< \beta \le 1$ and $p=\lceil \beta^{-1} \rceil -1$.
Throughout this paper, $\beta$ and $\beta n$ are assumed to be a rational and an integer respectively.
Note that $p$ is defined so that by Tur\'an's Theorem~\cite{MR0018405} $k_r(n,(1-\beta)n)>0$ for all $n$ (such that $\beta n$ is an integer) if and only if $r \le p+1$.
Since the case $\beta=1$ implies the trivial case $\delta =0$, we may assume that $0 < \beta< 1$.
Furthermore, we consider the cases $1/(p+1)\le \beta <1/p$ separately for positive integers~$p$.
Hence, the condition $p=2$ is equivalent to $1/3 \le \beta < 1/2$,  that is, $n/2 < \delta \le 2n/3$.

Next, we definite a family $\mathcal{G}(n,\beta)$ of graphs, which gives an upper bound on $k_r(n,\delta)$, where $\delta = (1-\beta)n$ and integers $r \ge 3$.

\begin{dfn} \label{dfn:G(n,b)}
Let $n$ and $(1-\beta) n$ be positive integers not both odd with $0 < \beta < 1$.
Define $\mathcal{G}(n,\beta)$ to be the family of graphs~$G =(V,E)$ of order~$n$ satisfying the following properties. 
There is a partition of~$V$ into $V_0, V_1, \dots, V_{p-1}$ with $|V_0| = (1-(p-1)\beta)n$ and $|V_i| = \beta n$ for~$1 \le i \le p-1$, where again $p = \lceil \beta^{-1} \rceil -1$.
For $0 \leq i < j \leq p-1$, the bipartite graph $G[V_i,V_j]$ induced by the vertex classes $V_i$ and $V_j$ is complete.
For $1 \leq i \leq p-1$, the subgraph $G[V_i]$ induced by $V_i$ is empty and $G[V_0]$ is a $(1-p \beta) n$-regular graph such that the number of triangles in $G[V_0]$ is minimal over all $(1-p \beta) n$-regular graphs of order~$|V_0| = (1-(p-1)\beta)n$.
\end{dfn}

Note that $\mathcal{G}(n,\beta)$ is only defined if $n$ and $(1-\beta)n$ are not both odd.
Thus, whenever we mention $\mathcal{G}(n,\beta)$, we automatically assume that $n$ or $(1-\beta)n$ is even.
Furthermore, we say $(n,\beta)$ is \emph{feasible} if $G[V_0]$ is triangle-free for $G \in \mathcal{G}(n,\beta)$.
Note that $G[V_0]$ is regular of degree $(1-p\beta)n \le (1-(p-1)\beta)n/2 = |V_0|/2$.
Thus, if $|V_0|$ is even, then $G[V_0]$ is triangle-free.
Therefore, for a given~$\beta$, there exist infinitely many choices of~$n$ such that $(n, \beta)$ is a feasible pair.
If $(n, \beta)$ is not a feasible pair, then $|V_0|$ is odd.
Moreover, it is easy to show that $k_3(G[V_0])= k_3^{reg}(n_0,\delta_0) = o(n^3)$, where $n_0 = |V_0| = (1-(p-1)\beta)n$, $\delta_0 = (1-p\beta)n$ and $k_r(H)$ is the number of $r$-cliques in a graph~$H$.

By Definition~\ref{dfn:G(n,b)}, every $G \in \mathcal{G}(n, \beta)$ is $(1-\beta) n$-regular.
In particular, for positive integers~$r \ge 3$, the number of $r$-cliques in $G$ is exactly 
\begin{align}
k_r(G) = & g_r(\beta)n^r + \binom{p-1}{r-3}(1-p\beta)^{r-3}n^{r-3} k_3(G[V_0]),\nonumber \\
 = & g_r(\beta)n^r + \binom{p-1}{r-3}(1-p\beta)^{r-3}n^{r-3} k_3^{reg}(n_0,\delta_0), \label{eqn:conjeqn}
\intertext{where $n_0 = (1-(p-1)\beta)n$, $\delta_0 = (1-p\beta)n$ and}
  g_r(\beta) = & \binom{p-1}{r}\beta^{r} + \binom{p-1}{r-1}(1-(p-1)\beta)\beta^{r-1} \nonumber \\  
	&+ \frac12\binom{p-1}{r-2}(1-p\beta)(1-(p-1)\beta) \beta^{r-2} \nonumber
\end{align}
with $\binom{x}{y}$ defined to be 0 if $x<y$ or~$y<0$.
Since  $k_3^{reg}(n_0,\delta_0) = o(n^3)$, \eqref{eqn:conjeqn} becomes $k_r(G) =  (g_r(\beta)+o(1)) n^r$.
In fact, most of the time, we consider the case when $(n, \beta)$ is feasible, i.e.~$k_3(G[V_0]) =0$ and $k_r(G) =  g_r(\beta)n^r$.
We conjecture that if $(n,\beta)$ is feasible then $\mathcal{G}(n,\beta)$ is the extremal family for $k_r(n,\delta)$ with $\delta = (1-\beta) n$ and $3 \le r \le p+1 = \lceil \beta^{-1} \rceil$.
\begin{conj} \label{conj}
Let $n$ and $\delta$ be positive integers.
Then $$k_r(n,\delta )  \ge g_r(\beta)n^r,$$ where $\delta = (1-\beta)n$ and $r \ge 3$.
Moreover, for $3\leq r\leq p+1 = \lceil \beta^{-1} \rceil$ equality holds if and only if $(n,\beta)$ is feasible and the extremal graphs are members of~$\mathcal{G}(n, \beta)$.
\end{conj}

By Tur\'an's Theorem~\cite{MR0018405}, the above conjecture is true when $p=1$ or $r > p+1$.
If $\beta  = 1/(p+1)$ and $(p+1) |n$, then $\mathcal{G}(n, 1/(p+1))$ only consists $T_{p+1}(n)$, the $(p+1)$-partite Tur\'an graph of order~$n$.
Bollob\'as~\cite{MR0424614} proved that if $(p+1)|n$ and $e =(1-1/(p+1))n^2/2$, then $f_r(n,e) = k_r(T_{p+1}(n))$.
Moreover, $T_{p+1}(n)$ is the only graph of order $n$ with $e$ edges and $f_r(n,e)$ $r$-cliques.
Hence, it is an easy exercise to show that Conjecture~\ref{conj} is true when $\beta  = 1/(p+1)$.

It should be noted that since $\mathcal{G}(n,\beta)$ defines a family of regular graphs, we also conjecture that $k_r^{reg}(n,\delta)$ is achieved by $G \in \mathcal{G}(n,\beta)$, where $\delta = (1-\beta)n$.
However, we do not address the problem $k_r^{reg}(n,\delta)$ here.
For the remainder of the paper, all graphs are also assumed to be of order~$n$ with minimum degree $\delta = (1-\beta)n$ unless stated otherwise.

\section{Main results} \label{sec:mainresult}
By our previous observation, Conjecture~\ref{conj} is true for the following three cases: $p=1$, $r>p+1$ and $\delta= (1-1/(p+1)) n$.
That leaves the situation when $3\leq r \leq p+1$ and $\delta > n/2$.
In Section~\ref{sec:p=2}, we prove Conjecture~\ref{conj} for $n/2 < \delta \le 2n/3$, as follows.
\begin{thm} \label{thm:p=2}
Let $n$ and $\delta$ be positive integers with $n/2 < \delta \le 2n/3$.
Then $$k_3(n,\delta ) \ge g_3(\beta)n^3,$$ where $\delta = (1-\beta)n$.
Moreover, equality holds if and only if $(n,\beta)$ is feasible and the extremal graphs are members of~$\mathcal{G}(n, \beta)$.
\end{thm}
The ideas in the proof, which is short, form the framework for our other results.
The next simplest case is that of $K_{p+2}$-free graphs. 
Notice that, by the definition of $p$, $G$ must contain $K_{p+1}$'s but need not contain $K_{p+2}$.
Conjecture~\ref{conj} is proved for $K_{p+2}$-free graphs by the next theorem.

\begin{thm} \label{thm:conjheavyfree}
Let $n$ and $\delta$ be positive integers.
Let $G$ be a $K_{p+2}$-free graph of order~$n$ with minimum degree $\delta$, where $\delta = (1-\beta)n$ and $p = \lceil \beta^{-1} \rceil -1$.
Then, $$k_r(G)  \ge g_r(\beta)n^r$$ for positive integers~$r$.
Moreover, for $3\leq r\leq p+1$ equality holds if and only if $(n,\beta)$ is feasible, and the extremal graphs are members of~$\mathcal{G}(n, \beta)$.
\end{thm}

Theorem~\ref{thm:conjheavyfree} is proved in Section~\ref{sec:kp+2-free}, after some notations and basic inequalities have been set up in Section~\ref{sec:DegreeofCliques}.
It shows that the difficult in proving Conjecture~\ref{conj} is in handling $(p+2)$-cliques.
We discuss this situation in Section~\ref{sec:p=3} for the case $p=3$, and by a detailed analysis of 5-cliques in Section~\ref{sec:p=3thmkpkp+1}, proving Conjecture~\ref{conj} for $2n/3\ < \delta \le 3n/4$, as follows.

\begin{thm} \label{thm:p=3conj2}
Let $n$ and $\delta$ be positive integers with $2n/3 < \delta \le 3n/4$.
Then $$k_r(n,\delta )  \ge g_r(\beta)n^r,$$ for positive integers $r$ and $\delta = (1-\beta)n$.
Moreover, for $3 \le r \le4$ equality holds if and only if $(n,\beta)$ is feasible and the extremal graphs are members of~$\mathcal{G}(n, \beta)$.
\end{thm}

This theorem is the hardest in the paper.
We have in fact proved Conjecture~\ref{conj} for $3n/4 < \delta \le 4n/5$ by a similar argument.
It is too complicated to be included in this paper, but it can be found in~\cite{lothesis}.
For each positive integer~$p \ge 5$, it is likely that by following the arguments in the proof of Theorem~\ref{thm:p=3conj2} 
one could construct a proof for Conjecture~\ref{conj} when $(1-1/p)n < \delta \le (1-1/(p+1))n$. 

We give two more results in support of Conjecture~\ref{conj} in Section~\ref{sec:partial} and Section~\ref{sec:r=p+1}.
The first is that for every positive integer~$p$, Conjecture~\ref{conj} holds for a positive proportion of values of~$\delta$.
\begin{thm} \label{thm:partial}
For every positive integer $p$, there exists a (calculable) constant $\epsilon_p>0$ so that if $n$ and $\delta$ are positive integers such that $(1-1/(p+1) - \epsilon_p)n < \delta \le (1-1/(p+1))n$, then $$k_r(n,\delta )  \ge g_r(\beta)n^r,$$ for positive integers~$r$ and $\delta = (1-\beta)n$.
Moreover, for $3 \le r \le p+1$ equality holds if and only if $(n,\beta)$ is feasible and the extremal graphs are members of~$\mathcal{G}(n, \beta)$.
\end{thm}

Finally, using a different argument, we can show that Conjecture~\ref{conj} holds in the case $r=p+1$ (the largest value of $r$ for which $r$-cliques are guaranteed).
\begin{thm} \label{thm:conjr=p+1}
Let $n$ and $\delta$ be positive integers.
Then $$k_{p+1}(n,\delta )  \ge g_{p+1}(\beta)n^{p+1},$$ where $\delta = (1-\beta)n$ and $p= \lceil \beta^{-1} \rceil -1 $.
Moreover, equality holds if and only if $(n,\beta)$ is feasible and the extremal graphs are members of~$\mathcal{G}(n, \beta)$.
\end{thm}

\section{Proof of Theorem~\ref{thm:p=2}} \label{sec:p=2}

Here we prove Theorem~\ref{thm:p=2}, that is Conjecture~\ref{conj} for $n/2 < \delta \le 2n/3$, so $1/3\le\beta < 1/2$ and $p=2$.

\begin{proof}[Proof of Theorem~\ref{thm:p=2}]
Let $G$ be a graph of order $n$ with minimum degree~$\delta$.
Since $G$ has at least $\delta n/2 = (1-\beta)n^2/2$ edges,
\begin{align*}
(1-2\beta) \beta n k_2(G) \ge (1-2\beta)(1-\beta)\beta n^3/2= g_3(\beta) n^3.
\end{align*}
Thus, in proving the inequality in Theorem~\ref{thm:p=2}, it is enough to show that $k_3(G) \ge (1-2\beta) \beta n k_2(G)$.

For an edge $e$, define $d(e)$ to be the number of triangles containing~$e$ and write $D(e) = d(e)/n$.
Clearly, $$n \sum_{e \in E(G)} D(e) = \sum d(e)= 3k_3(G).$$
In addition, $D(e) \ge 1-2\beta$ for each edge~$e$, because each vertex in $G$ misses at most $\beta n$ vertices.
Since $\beta <1/2$, $D(e)>0$ for all $e \in E(G)$ and so every edge is contained in a triangle.
Let $T$ be a triangle in~$G$.
Similarly, define $d(T)$ to be the number of 4-cliques containing $T$ and write $D(T)=d(T)/n$.
We claim that
\begin{align}
\sum_{e \in E(T)} D(e) \ge 2-3\beta + D(T). \label{eqn:sumk_3D(K2)basic}
\end{align}
Let $n_i$ be the number vertices in $G$ with exactly $i$ neighbours in~$T$ for $i= 0,1,2,3$.
Clearly, $n=n_0+n_1+n_2+n_3$.
By counting the number of edges incident with~$T$, we obtain 
\begin{align}
3(1-\beta)n \le \sum_{v \in V(T)} d(v) = 3n_3+2n_2+n_1 \le 2n_3+n_2+n. \label{eqn:sumk_3D(K2)basic3}
\end{align}
On the other hand, $n_3= d(T)$ and $n_2+3n_3 = \sum_{e \in E(G)} d(e)$.
Hence, \eqref{eqn:sumk_3D(K2)basic} holds.
Notice that if equality holds in \eqref{eqn:sumk_3D(K2)basic} then $d(v)=(1-\beta)n$ for all~$v \in T$.

For an edge $e$, define $D_-(e) = \min\{D(e), \beta \}$.
We claim that
\begin{align}
\sum_{e \in E(T)} D_-(e) \ge 2-3\beta  \label{eqn:sumk_3D(K2)basic5}
\end{align}
for every triangle~$T$.
If $D(e) =D_-(e)$ for each edge~$e$ in~$T$, then \eqref{eqn:sumk_3D(K2)basic5} holds by~\eqref{eqn:sumk_3D(K2)basic}.
Otherwise, there exists $e_0 \in E(T)$ such that $D(e_0)\ne D_-(e_0)$.
This means that $D_-(e_0)=\beta$.
Recall that for the other two edges~$e$ in~$T$, $D(e) \ge 1-2\beta$, so $\sum D_-(e) \ge \beta +2(1-2\beta) = 2-3\beta$.
Hence, \eqref{eqn:sumk_3D(K2)basic5} holds for every triangle~$T$.

Next, by summing~\eqref{eqn:sumk_3D(K2)basic5} over all triangles~$T$ in~$G$, we obtain 
\begin{align}
n \sum_{e \in E(G)} D_-(e)D(e) = \sum_{T} \sum_{e \in E(T)} D_-(e) \ge (2-3\beta)k_3(G). \label{eqn:sumk_3D(K2)basic2}
\end{align}
We are going to bound $\sum D_-(e)D(e)$ above in terms of $\sum D(e)$, which is equal to $3k_3(G)/n$, by the following proposition.

\begin{prp} \label{prp:keyprp}
Let $\mathcal{A}$ be a finite set. 
Suppose $f,g: \mathcal{A} \rightarrow \mathbb{R}$ with $f(a) \le M$ and $g(a) \ge m$ for all~$a \in \mathcal{A}$.
Then
\begin{align*}
 \sum_{a \in \mathcal{A}} f(a)g(a) \leq m\sum_{a \in \mathcal{A}} f(a) + M\sum_{a \in \mathcal{A}} g(a) - m M |\mathcal{A}|,
\end{align*}
with equality if and only if~for each~$a \in \mathcal{A}$,~$f(a) = M$ or~$g(a) = m$.
\end{prp}
\begin{proof}
Observe that~$\sum_{a \in \mathcal{A}} (M-f(a))(g(a)-m) \geq 0$.
\end{proof}

Recall that $D(e) \ge 1-2\beta$ and $D_-(e) \le \beta$.
By Proposition~\ref{prp:keyprp} taking $\mathcal{A} = E(G)$, $f = D_-$, $g=D$, $M=\beta$ and $m=1-2\beta$, we have
\begin{align}
	n \sum_{e \in E(G)} D(e)D_-(e) & \leq  (1-2\beta )n \sum_{e \in E(G)} D_-(e) + \beta n \sum_{e \in E(G)} D(e) -(1-2\beta)\beta n k_2(G) \nonumber \\
	& \leq (1-\beta)n  \sum_{e \in E(G)} D(e) -(1-2\beta)\beta n k_2(G) \label{eqn:sumk_3D(K2)equality} \\
	n \sum_{e \in E(G)} D(e)D_-(e)  & \leq  3 (1-\beta) k_3(G) -(1-2\beta)\beta n k_2(G). \label{eqn:sumk_3D(K2)basic4}
\end{align}
After substitution of~\eqref{eqn:sumk_3D(K2)basic4} into~\eqref{eqn:sumk_3D(K2)basic2} and rearrangement, we have $$k_3(G) \ge (1-2\beta)\beta k_2(G) n.$$
Thus, we have proved the inequality in Theorem~\ref{thm:p=2}.

Now suppose equality holds, i.e.~$k_3(G) = (1-2\beta)\beta k_2(G) n$.
This means that equality holds in~\eqref{eqn:sumk_3D(K2)equality}, so (since $\beta<1/2$) $D(e) =D_-(e)$ for all~$e \in E(G)$.
Because equality holds in~\eqref{eqn:sumk_3D(K2)basic5}, $\sum_{e \in E(T)} D(e) = 2-3\beta$ for triangles~$T$.
Hence, $D(T) = 0$ for every triangle~$T$ by~\eqref{eqn:sumk_3D(K2)basic}, so $G$ is $K_4$-free.
In addition, by the remark following~\eqref{eqn:sumk_3D(K2)basic}, $G$ is $(1-\beta)n$-regular, because every vertex lies in a triangle as $D(e) >0$ for all edges~$e$.
Since equality holds in Proposition~\ref{prp:keyprp}, either $D(e) = 1-2\beta$ or $D(e) = \beta$ for each edge~$e$.
Recall that equality holds for~\eqref{eqn:sumk_3D(K2)basic}, so every triangle~$T$ contains exactly one edge $e_1$ with $D(e_1) = \beta$ and two edges, $e_2$ and $e_3$, with $D(e_2) = D(e_3) = 1- \beta$.
Pick an edge~$e$ with $D(e) = \beta$ and let $W$ be the set of common neighbours of the end vertices of~$e$, so $|W| = \beta n$.
Clearly $W$ is an independent set, otherwise $G$ contains a~$K_4$.
For each $w \in W$, $d(w) = (1- \beta)n$ implies $N(w) = V(G) \backslash W$.
Therefore, $G[V(G) \backslash W]$ is $(1-2\beta) n$-regular.
If there is a triangle~$T$ in $G[V(G) \backslash W]$, then $T \cup w$ forms a $K_4$ for~$w \in W$.
This contradicts the assumption that $G$ is $K_4$-free, so $G[V(G) \backslash W]$ is triangle-free.
Hence, $G$ is a member of $\mathcal{G}(n,\beta)$ and $(n, \beta)$ is feasible.
Therefore, the proof is complete.

\end{proof}

\section{Degree of a clique} \label{sec:DegreeofCliques}

Denote the set of $t$-cliques in $G[U]$ by $\mathcal{K}_t(U)$ and write $k_r(U)$ for~$|\mathcal{K}_r(U)|$.
If $U = V(G)$, we simply write $\mathcal{K}_r$ and~$k_r$.

Define the \emph{degree}~$d(T)$ of a $t$-clique~$T$ to be the number of $(t+1)$-cliques containing~$T$.
In other words, $d(T) = |\{S \in \mathcal{K}_{t+1} : T \subset S\}|$.
If $t=1$, then $d(v)$ coincides with the ordinary definition of the degree for a vertex~$v$.
If $t=2$, then $d(uv)$ is the number of common neighbours of the end vertices of the edge~$uv$, that is the codegree of $u$ and~$v$.
Clearly, $\sum_{T \in \mathcal{K}_t} d(T) = (t+1) k_{t+1}$ for~$t \ge 1$.
For convenience, we write $D(T)$ to denote~$d(T)/n$.

Recall that $p= \lceil \beta^{-1} \rceil -1$ and $1/(p+1) \le \beta < 1/p$.
Let $G_0 \in \mathcal{G}(n,\beta)$ with $(n,\beta)$ feasible.
Let $T$ be a $t$-clique in $G_0$.
It is natural to see that there are three types of cliques according to $|T \cap V_0|$.
However, if we consider $d(T)$, then there are only two types.
To be precise 
\begin{align*}
	D(T) = \begin{cases}
	1-t\beta	& \text{if~$|V(T) \cap V_0|=0,1$ and}\\
	(p-t+1)\beta 	& \text{if~$|V(T) \cap V_0| =2$},
	\end{cases}
\end{align*}
for~$T \in \mathcal{K}_t(G_0)$ and~$2\le t \le p+1$.
Next, define the functions $D_+$ and $D_-$ as follows.
For a graph~$G$ with minimum degree $\delta =(1-\beta)n$, define
\begin{align*}
D_{-} (T) & = \min \{ D(T), (p-t+1)\beta \}, \textrm{ and}\\
D_+ (T) &= D(T) - D_-(T) = \max \{ 0, D(T) - (p-t+1)\beta\}
\end{align*}
for $T \in \mathcal{K}_t$ and $1 \le t \le p+1$.
We say that a clique~$T$ is \emph{heavy} if~$D_+(T) >0$.
The graph~$G$ is said to be \emph{heavy-free} if and only if $G$ does not contain any heavy cliques.
Now, we study some basic properties of $D(T)$, $D_-(T)$ and $D_+(T)$.

\begin{lma} \label{lma:D(S)basicproperties}
Let $0< \beta<1$ and $p= \lceil \beta^{-1} \rceil-1$.
Suppose $G$ is a graph of order~$n$ with minimum degree $(1-\beta)n$.
Suppose $S \in \mathcal{K}_{s}$ and $T \in \mathcal{K}_t(S)$ for $1 \leq t < s$.
Then 
\begin{enumerate}
	\item[(i)]$D(S) \geq 1- s\beta$,
	\item[(ii)]$D(S) \geq D(T) - (s-t)\beta$,
	\item[(iii)] for $s\le p+1$, $D_+(T) \leq D_+(S) \leq D_+(T) + (s-t)\beta$,
	\item[(iv)] if $T$ is heavy and $s\le p+1$ then $S$ is heavy,  and
	\item[(v)] if $T$ is not heavy and $s\le p+1$, then $D_+(S) \le  (s-t)\beta$. In particular, if $t=s-1 \le p$, then $D_+(S) \le  \beta$.
\end{enumerate}
Moreover, $G$ is $K_{p+2}$-free if and only if $G$ is heavy-free.
\end{lma}

\begin{proof}
For each $v \in S$, there are at most $\beta n$ vertices not joined to~$v$.
Hence, $D(S) \geq 1 - s \beta$, so $(i)$ is true.
Similarly, consider the vertices in~$S \backslash T$, so $(ii)$ is also true.
If $s\le p+1$ and $D_+(T)>0$, then we have
\begin{align*}
 D_+(S) + (p-s+1)\beta \ge & D(S) \\
\geq & D(T) - (s-t)\beta \\=& D_+(T) + (p-t+1)\beta - (s-t)\beta,
\end{align*}
so the left inequality of $(iii)$ is true.
Since $D(S) \leq D(T)$, the right inequality of $(iii)$ is also true by the definition of $D_+(S)$ and $D_+(T)$.
Hence, $(iv)$ and $(v)$ are true by the left and right inequality in $(iii)$ respectively.
Notice that $D(U) = D_+(U)$ for $U \in \mathcal{K}_{p+1}$.
Hence, by $(iv)$, $G$ is ${K}_{p+2}$-free if and only if $G$ is heavy-free.
\end{proof}

Now we prove the generalised version of~\eqref{eqn:sumk_3D(K2)basic}, that is, the sum of degrees of $t$-subcliques in a $s$-clique.

\begin{lma} \label{lma:sumk_tbasic}
Let $0< \beta<1$.
Let $s$ and~$t$ be integers with $2 \leq t < s$.
Suppose $G$ is a graph of order $n$ with minimum degree $(1-\beta)n$.
Then
\begin{align}
	\sum_{T \in \mathcal{K}_{t}(S)} D(T) & \geq (1-\beta) s\binom{s-2}{t-1} - (t-1) \binom{s-1}{t} +  \binom{s-2}{t-2} D(S) \nonumber 
\end{align}
for $S \in \mathcal{K}_{s}$.
Moreover, if equality holds, then $d(v) =(1-\beta)n$ for all~$v \in S$.
\end{lma}

\begin{proof}
Let $n_i$ be the number of vertices with exactly $i$ neighbours in~$S$.
The following three equations :
\begin{align}
	\sum_i n_i  & =  n, \label{eqn:(0)} \\
	\sum_i i n_i & =  \sum_{v \in V(S)} d(v) \geq s (1-\beta) n, \label{eqn:(1)} \\
	\sum_i \binom{i}{t} n_i & =  \sum_{T \in \mathcal{K}_{t}(S)} D(T) n, \label{eqn:(t)}
\end{align}
follow from a count of the number of vertices, edges and $(t+1)$-cliques respectively.
Next, by considering $(t-1) \binom{s-1}{t}\eqref{eqn:(0)} - \binom{s-2}{t-1}\eqref{eqn:(1)} +~\eqref{eqn:(t)}$, we have
\begin{align*}\sum_{T \in \mathcal{K}_{t}(S)} D(T)n  \geq \left( (1-\beta) s\binom{s-2}{t-1} - (t-1) \binom{s-1}{t} \right)n + \sum_{0 \leq i \leq s} x_i n_i,
\end{align*}
where $x_i = \binom{i}{t} +(t-1)\binom{s-1}{t} - i\binom{s-2}{t-1}$.
Notice that $ x_i = x_{i+1}+\binom{s-2}{t-1} - \binom{i}{t-1} \geq x_{i+1}$ for~$0 \leq i \leq s-2$.
For~$i=s-1$, we have
\begin{align*}
	x_{s-1} & = \binom{s-1}{t} +(t-1)\binom{s-1}{t} - (s-1)\binom{s-2}{t-1}\\
	& = t \binom{s-1}{t} - (s-1)\binom{s-2}{t-1} = 0.
\end{align*} 
For $i=s$, $n_s = D(S)n$ and
\begin{align*}
	 x_{s} = &\binom{s}{t} +(t-1)\binom{s-1}{t} - s\binom{s-2}{t-1}
	 \\
	= & t\binom{s-1}{t} + \binom{s-1}{t-1} - s\binom{s-2}{t-1}\\
	 = &(s-t+1)\binom{s-1}{t-1} - s\binom{s-2}{t-1} \\
	= & (s-t+1)\binom{s-2}{t-2} - (t-1)\binom{s-2}{t-1}  \\ = &\binom{s-2}{t-2}.
\end{align*}
In particular, if equality holds in the lemma, then equality holds in~\eqref{eqn:(1)}.
This means that $d(v) = (1-\beta)n$ for all~$v \in S$.
\end{proof}

Most of the time, we are only interested in the case when~$s=t+1$.
Hence, we state the following corollary.

\begin{cor} \label{cor:sumk_tbasic}
Let $0< \beta<1$.
Suppose $G$ is a graph or order $n$ with minimum degree $(1-\beta)n$.
Then
\begin{align}
	\sum_{T \in \mathcal{K}_{t}(S)} D(T) & \geq 2-(t+1)\beta +  (t-1) D(S) \nonumber 
\end{align}
for $S \in \mathcal{K}_{t+1}$ and integer~$t\ge 2$.
Moreover, if equality holds, then $d(v) =(1-\beta)n$ for all~$v \in S$. 
$\hfill{\square}$
\end{cor}

In the next lemma, we show that the functions~$D$ in Lemma~\ref{lma:sumk_tbasic} can be replaced with~$D_-$.

\begin{lma} \label{lma:sumk_tbasicimprove}
Let $0< \beta <1$ and $p = \lceil \beta^{-1} \rceil -1$.
Let $s$ and~$t$ be integers with $2 \leq t < s \leq p+1$.
Suppose $G$ is a graph of order~$n$ with minimum degree~$(1-\beta) n$.
Then, for $S \in \mathcal{K}_s$
\begin{align}
	\sum_{T \in \mathcal{K}_{t}(S)} D_-(T) & \geq (1-\beta) s\binom{s-2}{t-1} - (t-1) \binom{s-1}{t} +  \binom{s-2}{t-2} D_-(S). \nonumber
\end{align}
\end{lma}

\begin{proof}
Since $D_+(S) \geq D_+(T)$ for every~$T \in \mathcal{K}_t(S)$ by Lemma~\ref{lma:D(S)basicproperties}~$(iii)$, there is nothing to prove by Lemma~\ref{lma:sumk_tbasic} if there are at most $\binom{s-2}{t-2}$ heavy $t$-cliques in~$S$.
Now suppose there are more than $\binom{s-2}{t-2}$ heavy~$t$-cliques in~$S$.
In particular, $S$ contains a heavy $t$-clique, so $S$ is itself heavy with $D_-(S) = (p+1-s)\beta$ by Lemma~\ref{lma:D(S)basicproperties}~$(iv)$.
Thus, the right hand side of the inequality is $\binom{s}{t}(1-t\beta) + \binom{s-2}{t-2} ((p+1)\beta-1)$.
By Lemma~\ref{lma:D(S)basicproperties}~$(i)$ we have that $D_-(T) \ge (1-t\beta)$ for~$T\in \mathcal{K}_t(S)$.
Furthermore, by Lemma~\ref{lma:D(S)basicproperties}~$(iv)$ $D_-(T) = (p-t+1)\beta$ if~$T$ is heavy, so summing~$D_-(T)$ over~$T \in \mathcal{K}_t(S)$ gives
 \begin{align}
	\sum_{T \in \mathcal{K}_{t}(S)} D_-(T) & \geq k_t^+(S) (p-t+1)\beta + \left( \binom{s}{t} - k_t^+(S) \right) (1-t\beta) \nonumber \\
	& =\binom{s}{t}(1-t\beta) + k_t^+(S) ((p+1)\beta-1). \nonumber
\end{align}
This completes the proof of the lemma. 
\end{proof}

Define the function $\widetilde{D} : \mathcal{K}_{t+1} \rightarrow \mathbb{R}$ such that
\begin{align*}
\widetilde{D}(S) = \sum_{T \in \mathcal{K}_{t}(S)} D_-(T) - \Big( 2-(t+1)\beta +  (t-1) D_-(S) \Big)
\end{align*}
for $S \in \mathcal{K}_{t+1}$ and $2\le t \le p$.
Hence, for $s=t+1$, Lemma~\ref{lma:sumk_tbasicimprove} gives the following corollary.

\begin{cor} \label{cor:sumk_tbasicimproved}
Let $0< \beta <1$ and $p = \lceil \beta^{-1} \rceil -1$.
Let $t$ be integer with $2 \le t  \le p$.
Suppose $G$ is a graph of order~$n$ with minimum degree~$(1-\beta) n$.
Then~$\widetilde{D}(S) \ge 0$ for $S \in \mathcal{K}_{t+1}$.
\hfill{$\square$}
\end{cor} 

Next, we bound $\sum_{S\in \mathcal{K}_{t+1}} \widetilde{D}(S)$ from above using Proposition~\ref{prp:keyprp}.

\begin{lma} \label{lma:sumwidetildeD(S)}
Let $0< \beta <1$ and $p = \lceil \beta^{-1} \rceil -1$.
Let $t$ be an integer with $2 \le t  \le p$.
Suppose $G$ is a graph of order~$n$ with minimum degree~$(1-\beta) n$.
Then
\begin{align*}
  \sum_{S\in \mathcal{K}_{t+1}} \widetilde{D}(S) 
 \le  \big(t-1+(p-2t+2)(t+1)\beta \big)k_{t+1} 
+ (t-1) \sum_{S \in \mathcal{K}_{t+1}} D_+(S) 
\\  - (1-t\beta) (p-t+1)\beta n k_t
 - (t-1)(t+2) \frac{k_{t+2}}{n} 
-(1-t\beta)n \sum_{T \in \mathcal{K}_{t}} D_+(T).
\end{align*}
Moreover, equality holds if and only if for each~$T \in \mathcal{K}_{t}$, either~$D_-(T) = 1-t\beta$ or~$D_-(T) = (p-t+1)\beta$.
\end{lma}

\begin{proof}
Notice that the sum $\widetilde{D}(S)$ over $S \in \mathcal{K}_{t+1}$ is equal to
\begin{align}
  & \sum_{S \in \mathcal{K}_{t+1}} \sum_{T \in \mathcal{K}_{t}(S)} D_-(T) - (2-(t+1)\beta)k_{t+1} - (t-1) \sum_{S \in \mathcal{K}_{t+1}}   D_-(S). \label{eqn:sumwidetildeD(S)}
\end{align}
Consider each each term separately.
Since $D(S)=D_-(S)+D_+(S)$,
\[\sum_{S \in \mathcal{K}_{t+1}} D_-(S) =  \sum_{S \in \mathcal{K}_{t+1}} D(S)- \sum_{S \in \mathcal{K}_{t+1}} D_+(S) = \frac{(t+2) k_{t+2}}n - \sum_{S \in \mathcal{K}_{t+1}} D_+(S).\]
By interchanging the order of summations, we have 
\begin{align*}
 &\sum_{S \in \mathcal{K}_{t+1}} \sum_{T \in \mathcal{K}_{t}(S)} D_-(T) = n \sum_{T \in \mathcal{K}_t} D_-(T)D(T),
\intertext{and by Proposition~\ref{prp:keyprp} taking $\mathcal{A} = \mathcal{K}_t$, $f=D_-$, $g=D$, $M=(p-t+1)\beta$ and $m=1-t\beta$}
	\le & (1-t\beta) n \sum_{T \in \mathcal{K}_{t}} D_-(T) + (p-t+1)\beta n \sum_{T \in \mathcal{K}_{t}} D(T) -(1-t\beta)(p-t+1)\beta n k_t\\
= &  (1+(p-2t+1)\beta)n \sum D(T) - (1-t\beta)n \sum D_+(T)  -(1-t\beta)(p-t+1)\beta n k_t\\
= & (1+(p-2t+1)\beta)(t+1)k_{t+1} - (1-t\beta)n \sum  D_+(T)-(1-t\beta)(p-t+1)\beta n k_t.
\end{align*}
Hence, substituting these identities back into~\eqref{eqn:sumwidetildeD(S)}, we obtain the desired inequality in the lemma.

By Proposition~\ref{prp:keyprp}, equality holds if and only if for each~$T\in\mathcal{K}_t$, either $D(T)=1-t\beta$ or $D_-(T)=(p-t+1)\beta$.
\end{proof}

To keep our calculations simple, we are going to establish a few relationships between $g_t(\beta)$ and $g_{t+1}(\beta)$ in the next lemma.

\begin{lma} \label{lma:property-g}
Let $0< \beta <1$ and $p = \lceil \beta^{-1} \rceil -1$.
Let $t$ be an integer with $2 \leq t \leq p$.
Then
\begin{align}
	(t+1) g_{t+1}(\beta)  = & (1-t\beta) g_t(\beta) \nonumber \\& +  \frac12 \binom{p-1}{t-2} ((p+1)\beta-1) (1-(p-1)\beta) (1-p\beta) \beta^{t-2} \label{eqn:property-g1},\\
	g_{t+1}(\beta) =&  \frac{ (1-t\beta)(p-t+1)\beta g_t(\beta)+ (t-1)(t+2)g_{t+2}(\beta)}{t-1+(t+1)(p-2t+2)\beta}  \label{eqn:property-g2}.
\end{align}
Moreover
\begin{align}
	\frac{ g_p(\beta) }{ g_{p+1}(\beta)} = \frac1\beta\left( 1+ \frac{ \beta g_{p-1}(\beta') }{(1-\beta) g_{p}(\beta')} \right), \label{eqn:property-g3}
\end{align}
 where~$\beta' = \beta /(1-\beta)$.
\end{lma}

\begin{proof}
We fix $\beta$ (and $p$) and write $g_t$ to denote~$g_t(\beta)$.
Pick $n$ such that $(n, \beta)$ is feasible and let $G \in \mathcal{G}(n, \beta)$ with partition classes $V_0, V_1,\dots, V_{p-1}$ as described in Definition~\ref{dfn:G(n,b)}.
Thus, for $T \in \mathcal{K}_{t}$, $D(T) = 1-t\beta$ or~$D(T) = (p-t+1)\beta$.
Since $D(T) = (p-t+1)\beta$ if and only if $|V(T) \cap V_0|=2$, there are exactly
\begin{align}
\frac12 \binom{p-1}{t-2} (1-(p-1)\beta) (1-p\beta) \beta^{t-2} n^{t} \nonumber
\end{align}
$t$-cliques~$T$ with $D(T) = (p-t+1)\beta$.
Also, we have
\begin{align*}
(t+1) g_{t+1} n^{t+1}  = (t+1) k_{t+1} = n \sum_{T \in \mathcal{K}_t} D(T).
\end{align*}
Hence, \eqref{eqn:property-g1} is true, by expanding the right hand side of the above equation.
For $2 \leq s < p$, let $f_{s}$ and $f_{s+1}$ be \eqref{eqn:property-g1} with $t=s$ and $t=s+1$ respectively.
Then $\eqref{eqn:property-g2}$ follows by considering $ (p-s+1)f_s - (s-1)\beta f_{s+1}$.

Now let $G' = G \backslash V_{p-1}$.
Notice that $G'$ is $(1-2\beta)n$-regular with $(1-\beta)n$ vertices.
We observe that $G'$ is a member of $\mathcal{G}(n',\beta')$, where $n' = (1-\beta)n$ and $\beta' = \beta /(1-\beta)$.
Observe that $\lceil \beta'^{-1} \rceil-1 = p-1$, so $1/p\le \beta'<1/(p-1)$.
Recall that $k_t(G) = g_t(\beta) n^t$ for all $2 \leq t \leq p$, so~$k_{p+1}(G) g_{p}(\beta) = k_p(G) g_{p+1}(\beta) n$.
Similarly,~$k_{p}(G') g_{p-1}(\beta') = k_{p-1}(G') g_{p}(\beta') n$.
By considering $\mathcal{K}_p(G)$ and $\mathcal{K}_{p+1}(G)$, we obtain the following two equations :
\begin{align}
	k_{p+1}(G) & = \beta n k_{p}(G'), \label{eqn:g_t1} \\
	k_p(G) & = \beta n k_{p-1}(G') + k_{p}(G') = \beta n \frac{g_{p-1}(\beta') k_p(G')}{n'g_{p}(\beta')} + k_{p}(G') \nonumber \\
	& = \left( 1+ \frac{ \beta g_{p-1}(\beta') }{(1-\beta) g_{p}(\beta')} \right) k_p(G') \label{eqn:g_t2}.
\end{align}
By substituting \eqref{eqn:g_t1} and \eqref{eqn:g_t2} into $k_{p}(G) n  / k_{p+1}(G) =  g_{p}(\beta)/ g_{p+1}(\beta)$, we obtain~\eqref{eqn:property-g3}.
The proof is complete.
\end{proof}

\section{$K_{p+2}$-free graphs} \label{sec:kp+2-free}

In this section, all graphs are assumed to be $K_{p+2}$-free. 
Lemma~\ref{lma:D(S)basicproperties} implies that these graphs are also heavy-free.
This means that $D_+(T)=0$ and $D(T) \le (p-t+1)\beta$ for all~$T \in \mathcal{K}_t$ and $t \le p+1$.
We prove the theorem below, which easily implies Theorem~\ref{thm:conjheavyfree} as $g_2(\beta)n^2 = (1-\beta)n^2/2 \le k_2(G)$.

\begin{thm} \label{thm:conjforkp+2free}
Let $0< \beta <1$ and $p = \lceil \beta^{-1} \rceil -1$.
Suppose $G$ is a $K_{p+2}$-free graph of order~$n$ with minimum degree~$(1-\beta) n$.
Then
\begin{align}
	\frac{k_s(G)}{g_s(\beta)n^s} & \geq \frac{k_t(G)}{g_t(\beta) n^t} \label{eqn:conj}
\end{align}
holds for~$2\leq t < s \leq p+1$. 
Moreover, the following three statements are equivalent:
\begin{itemize}
   \item[(i)] Equality holds for some~$2\leq t < s \leq p+1$.
   \item[(ii)] Equality holds for all~$2\leq t < s \leq p+1$.
   \item[(iii)] The pair~$(n,\beta)$ is feasible and~$G$ is a member of~$\mathcal{G}(n, \beta)$.
\end{itemize}
\end{thm}

\begin{proof}
Fix $\beta$ and write $g_t$ to denote~$g_t(\beta)$.
Recall that $D_+(T) = 0 $ for cliques~$T$.
By Corollary~\ref{cor:sumk_tbasicimproved} and Lemma~\ref{lma:sumwidetildeD(S)},
we have 
\begin{align}
	k_{t+1} \geq  & \frac{(1-t\beta) (p-t+1)\beta n k_t + (t-1)(t+2) k_{t+2}/n}{t-1+(p-2t+2)(t+1)\beta} \label{eqn:ktkt+1kt+2inequality}
\end{align}
First, we are going to prove~\eqref{eqn:conj}.
It is sufficient to prove the case when~$s =t+1$.
We proceed by induction on $t$ from above.
For~$t=p$, $k_{p+2}=0$ and so \eqref{eqn:ktkt+1kt+2inequality} becomes
\begin{align*}
	(p-1-(p-2)(p+1)\beta)k_{p+1} & \geq (1-p\beta)\beta n k_p.
\end{align*}
Since~$g_{p+2}=0$, we have $k_{p+1}/g_{p+1}n^{p+1} \geq k_p/ g_p n^p$ by~\eqref{eqn:property-g2}.
Hence, \eqref{eqn:conj} is true for~$t=p$. 
For~$t<p$, \eqref{eqn:ktkt+1kt+2inequality} becomes
\begin{align}
	& (t-1+(t+1)(p-2t+2)\beta)k_{t+1} \nonumber \\ 
	\geq & (1-t\beta) (p+1-t)\beta n k_t + (t-1)(t+2) k_{t+2}/n \nonumber \\
\intertext{by the induction hypothesis}
	\geq & (1-t\beta) (p+1-t)\beta n k_t + (t-1)(t+2) g_{t+2} k_{t+1}/g_{t+1}.\label{eqn:induction} 
\end{align}
Thus, \eqref{eqn:conj} follows from~\eqref{eqn:property-g2}.

It is clear that $(iii)$ implies both $(i)$ and $(ii)$ by Definition~\ref{dfn:G(n,b)} and the feasibility of~$(n,\beta)$.
Suppose $(i)$ holds, so equality holds in \eqref{eqn:conj} for $t=t_0$ and $s=s_0$ with~$t_0< s_0$.
We claim that equality must also hold for $t=p$ and~$s= p+1$.
Suppose the claim is false and equality holds for~$t=t_0$ and~$s=s_0$, where~$s_0$ is maximal.
Since equality holds for $t=t_0$, by~\eqref{eqn:conj}, equality holds for $t = t_0, \dots, s_0-1$ with~$s=s_0$.
We may assume that $t=s_0-1$ and $s_0 \ne p+1$ and $k_{s_0+1}/g_{s_0+1} n  > k_{s_0}/g_{s_0}$.
However, this would imply a strictly inequality in~\eqref{eqn:induction} contradicting the fact that equality holds for $s=s_0$ and~$t=s_0-1$.
Thus, the proof of the claim is complete, that is, if $(i)$ holds then equality holds in \eqref{eqn:conj} for $t=p$ and~$s=p+1$.

Therefore, in order to prove that $(i)$ implies $(iii)$, it is sufficient to show that if $k_{p+1}/g_{p+1}n^{p+1} = k_p/g_pn^p$, then $(n,\beta)$ is feasible and $G$ is a member of~$\mathcal{G}(n, \beta)$.
We proceed by induction on~$p$.
It is true for $p=2$ by Theorem~3, so we may assume~$p\geq 3$.
Since equality holds in~\eqref{eqn:conj}, we have equality in
\eqref{eqn:ktkt+1kt+2inequality}, Corollary~\ref{cor:sumk_tbasicimproved} and Lemma~\ref{lma:sumwidetildeD(S)}.
Since $D_+$ is a zero function, equality in Corollary~\ref{cor:sumk_tbasicimproved} implies equality in Corollary~\ref{cor:sumk_tbasic} and so $G$ is $(1-\beta)n$-regular as every vertex is a $(p+1)$-clique.
In addition, for each $T \in \mathcal{K}_p$, either $D(T) = 1-p\beta$ or $D(T) =\beta$ by equality in Lemma~\ref{lma:sumwidetildeD(S)}.
Moreover, Corollary~\ref{cor:sumk_tbasic} implies that $\sum_{T \in \mathcal{K}_p(S)}D(T)  \ge 2-(p+1)\beta$ for $S \in \mathcal{K}_{p+1}$.
Thus, there exists $T \in \mathcal{K}_p(S)$ with $D(T) = \beta$.
Pick $T \in \mathcal{K}_p$ with $D(T) = \beta$ and let $W = \bigcap \{ N(v) : v \in V(S) \}$, so $|W| = \beta$.
Since $G$ is $K_{p+2}$-free, $W$ is a set of independent vertices.
For each $w \in W$, $d(w) = (1-\beta)n$, so $N(w) = V(G) \backslash W$.
Thus, the graph $G' = G [V(G) \backslash W]$ is $(1-\beta')n'$-regular, where $n' = (1-\beta)n$, $\beta' n' = (1-2\beta)n$ and $\beta' = \beta/(1-\beta)$.
Note that $\lceil \beta'^{-1} \rceil -1 = p-1$.
Since $G$ is~$K_{p+2}$-free, $G'$ is $K_{p+1}$-free.
Also, $k_{p+1}(G)  = \beta n k_{p}(G')$ and
\begin{align}
	k_{p}(G) = &\beta n k_{p-1}(G') + k_{p}(G') 
	 \overset{\text{by~\eqref{eqn:conj}}}{\leq} \beta \frac{g_{p-1}(\beta') k_p(G')}{g_p(\beta')} +  k_{p}(G') \nonumber \\
	=& \left(1+ \beta \frac{g_{p-1}(\beta')}{(1-\beta)g_p(\beta')} \right) k_{p}(G') 
	\overset{\text{by~\eqref{eqn:property-g3}}}{=} \frac{g_p(\beta) \beta}{g_{p+1}(\beta)} k_{p}(G')\label{eqn:conjforkp+2free2}. 
\end{align}
Hence, 
\begin{align*}
	g_p(\beta) \beta n k_{p}(G') & = g_p(\beta) k_{p+1}(G)  
	 \overset{\text{by }\eqref{eqn:conj}}{=} g_{p+1}(\beta) n k_p(G) 
	 \overset{\text{by }\eqref{eqn:conjforkp+2free2}}{\leq} g_p(\beta) \beta n k_{p}(G').
\end{align*}
Therefore, we have $k_p(G') / g_{p}(\beta') n'^p = k_{p-1}(G') / g_{p-1}(\beta') n'^{p-1}$.
By the induction hypothesis, $G' \in \mathcal{G}(n',\beta')$, which implies $G \in \mathcal{G}(n,\beta)$.
This completes the proof of the theorem.
\end{proof}

\section{$k_r(n,\delta)$ for~$2n/3< \delta \le  3n/4$} \label{sec:p=3}

By Theorem~\ref{thm:conjheavyfree}, in order to prove Conjecture~\ref{conj} it remains to handle the heavy cliques.
However, even though both Corollary~\ref{cor:sumk_tbasicimproved} and Lemma~\ref{lma:sumwidetildeD(S)} are sharp by considering $G \in \mathcal{G}(n,\beta)$, they are not sufficient to prove Conjecture~\ref{conj} even for the case when $2n/3 < \delta \le 3n/4$ by the observation below.
Let $2n/3 < \delta \le 3n/4$, $1/4 \le \beta <1/3$ and $p=3$.
By Corollary~\ref{cor:sumk_tbasicimproved} and Lemma~\ref{lma:sumwidetildeD(S)}, we have 
\begin{align}
 (1+3\beta)k_3+\sum_{T \in \mathcal{K}_3} D_+(T) & \ge 2(1-2\beta)\beta n k_2 +\frac4n k_4 + (1-2\beta)n \sum_{e \in \mathcal{K}_2} D_+(e),\label{eqn:p=3t=2}\\
(2-4\beta)k_4+2\sum_{S \in \mathcal{K}_4} D_+(S) & \ge (1-3\beta)\beta n k_3 + \frac{10}nk_5+(1-3\beta)n \sum_{T \in \mathcal{K}_3} D_+(T)\label{eqn:p=3t=3},
\end{align}
for $t=2$ and~$t=3$ respectively.
Since $D_-$ is a zero function on 4-cliques, $$\sum_{S \in \mathcal{K}_4} D_+(S)=\sum_{S \in \mathcal{K}_4} D(S) = 5k_5/n.$$ 
Hence, the terms with $k_5$ and $\sum D_+(S)$ cancel in~\eqref{eqn:p=3t=3}.
Also, $(1-2\beta)>0$, so we may ignore the term with~$\sum D_+(e)$ in~\eqref{eqn:p=3t=2}.
Recall that $g_2(\beta) = (1-\beta)/2$ and $g_3(\beta) =  (1-2\beta)^2\beta$.
After substitution of~\eqref{eqn:p=3t=3} into~\eqref{eqn:p=3t=2} replacing the $k_4$ term and rearrangement, we get
\begin{align*}
k_3(G) \ge g_3(\beta)n^3 - \frac{4\beta-1}{1-\beta}\sum_{T \in \mathcal{K}_3} D_+(T).
\end{align*}
However, $(4\beta-1) \ge0$ only if $\beta = 1/4$.
Hence, we are going to strengthen both \eqref{eqn:p=3t=3} and~\eqref{eqn:p=3t=2}.
Recall that \eqref{eqn:p=3t=2} is a consequence of  Corollary~\ref{cor:sumk_tbasicimproved} and Lemma~\ref{lma:sumwidetildeD(S)} for $t=2$.
Therefore, the following lemma, which is a strengthening of Corollary~\ref{cor:sumk_tbasicimproved} for $t=2$, would lead to a strengthening of~\eqref{eqn:p=3t=2}.

\begin{lma} \label{lma:p=3sumk_3}
Let $1/4\le \beta <1/3$.
Suppose~$G$ is a graph of order~$n$ with minimum degree~$(1-\beta) n$.
Then, for~$T \in \mathcal{K}_{3}$
\begin{align}
\widetilde{D}(T)  & \geq \left( 1-\frac2{29-75\beta}\right) \frac{4\beta-1}{1-2\beta}D_+(T) - (1-2\beta) \sum_{e \in \mathcal{K}_{2}(T)}  \frac{D_+(e)}{D_+(e)+\beta}\label{eqn:p=3t=2sumk_tstrengthen}.
\end{align}
Moreover, if equality holds then $T$ is not heavy and $d(v) = (1-\beta)n$ for all~$v \in T$.
\end{lma}

\begin{proof}
Let $c$ be $(1-2/(29-75\beta))(4\beta-1)/(1-2\beta)$.
Corollary~\ref{cor:sumk_tbasicimproved} gives $\widetilde{D}(T) \geq 0$, so we may assume that~$T$ is heavy.
In addition, Corollary~\ref{cor:sumk_tbasic} implies that
\begin{align}
	\widetilde{D}(T) +  \sum_{e \in \mathcal{K}_{2}(T)} D_+(e) & \geq D_+(T). \label{eqn:p=3t=2tildeD(T)}
\end{align}
Since $c < 1$, we may further assume that $T$ contains at least one heavy edge or else \eqref{eqn:p=3t=2sumk_tstrengthen} holds as \eqref{eqn:p=3t=2tildeD(T)} becomes $\widetilde{D}(T) \geq D_+(T) > cD_+(T)$.
Let~$e_0 \in \mathcal{K}_{2}(T)$ with~$D_+(e_0)$ maximal.
By substituting~\eqref{eqn:p=3t=2tildeD(T)} into~\eqref{eqn:p=3t=2sumk_tstrengthen}, it is sufficient to show that the function
\begin{align*}
f= \left(1-\frac{1-2\beta}{D_+(e_0)+2\beta} \right)  \widetilde{D}(T) - \left(c -  \frac{1-2\beta}{D_+(e_0)+2\beta} \right) D_+(T)
\end{align*}
is non-negative.

First consider the case when $D_+(T) \le 1- 3\beta$.
Lemma~\ref{lma:D(S)basicproperties}~$(iii)$ implies $D_+(e_0)  \le D_+(T) \le  1-3\beta$.
Hence,
\begin{align*}
  \frac{1-2\beta}{D_+(e_0)+2\beta} - c
\ge &   \frac{1-2\beta}{1-\beta} - c>0.
\end{align*}
Also, $1-2\beta \le 2\beta < D_+(e_0) +2\beta$.
Therefore,~$f> 0$ by considering the coefficients of $\widetilde{D}(T)$ and~$D(T)$.
Hence, we may assume~$D_+(T)>  1-3\beta$.
Since $T$ is heavy, $D_-(T) = \beta$.
Therefore, by the definition of~$\widetilde{D}$, we have
\begin{align}
 \widetilde{D}(T) = \sum_{e \in \mathcal{K}_2(T)} D_-(e) -2(1-\beta). \label{eqn:p=3t=2widetildeD(T)}
\end{align}
We split into different cases separately depending on the number of heavy edges in~$T$.

Suppose all edges are heavy.
Thus, $\widetilde{D}(T) = 2(4\beta-1)$ by~\eqref{eqn:p=3t=2widetildeD(T)}, because $D_-(e) = 2 \beta$ for all edges~$e$ in~$T$.
Clearly $D_+(T) = D(T) - \beta \le 1-\beta$.
Hence, \eqref{eqn:p=3t=2sumk_tstrengthen} is true as 
\begin{align*}
 \widetilde{D}(T) = 2(4\beta-1) \ge (4\beta-1)(1-\beta)/(1-2\beta) \ge cD_+(T).
\end{align*}
Thus, there exists an edge in~$T$ that is not heavy and $D_+(T) \le \beta$ by Lemma~\ref{lma:D(S)basicproperties}~$(v)$.

Suppose $T$ contains one or two heavy edges.
We are going to show that in both cases $$\widetilde{D}(T) \geq 2(D_+(T) - (1-3\beta)).$$
First assume that there is exactly one heavy edge in~$T$.
Let $e_1$ and $e_2$ be the two non-heavy edges in~$T$.
Note that $D_-(e_i) = D(e_i) \geq D(T) = D_+(T) + \beta$ for~$i=1,2$.
Thus, \eqref{eqn:p=3t=2widetildeD(T)} and 
Lemma~\ref{lma:D(S)basicproperties} imply that $\widetilde{D}(T) \geq 2(D_+(T) - (1-3\beta))$.
Assume that $T$ contains two heavy edges.
Let $e_1$ be the non-heavy edge in~$T$.
Similarly, we have $D_-(e_1) \geq D_+(T) + \beta$.
Recall that $D_+(T)\le \beta$, so \eqref{eqn:p=3t=2widetildeD(T)} and Lemma~\ref{lma:D(S)basicproperties} imply 
\begin{align*}
	\widetilde{D}(T) \ge & (4\beta + D_+(T) -(1-3\beta))\\
	 = & 4\beta -1 + D_+(T) - (1-3\beta) \geq 2(D_+(T) - (1-3\beta)).
\end{align*}
Since $\widetilde{D}(T) \geq 2(D_+(T) - (1-3\beta))$, in proving~\eqref{eqn:p=3t=2sumk_tstrengthen}, it is enough to show that
\begin{align}
	D(e_0) f = & (D_+(e_0)+2\beta)f \nonumber \\
	\ge & 2(D_+(e_0)+4\beta-1)(D_+(T) - (1-3\beta)) \nonumber\\ 
	& -  \left((D_+(e_0)+2\beta)c -  (1-2\beta) \right)D_+(T) \label{eqn:p=3D(e_0)f}
\end{align}
is non-negative for $0 < D_+(e_0) \leq D_+(T)$ and $1-3\beta \leq D_+(T) \leq \beta$.
Notice that for a fixed~$D_+(T)$ it is enough to check the boundary points of~$D_+(e_0)$.
For~$D_+(e_0) =0$, we have
\begin{align*}
	D(e_0)f \ge & (2(3- c)\beta -1)D_+(T)-2(4\beta-1)(1-3\beta) \\
	\ge& (4\beta-1)(D_+(T) - (1-3\beta)) > 0.
\end{align*}
For~$D_+(e_0)=D_+(T)$, the right hand side of \eqref{eqn:p=3D(e_0)f} becomes a quadratic function in $D_+(T)$.
Moreover, both coefficients of $D_+(T)^2$ and $D_+(T)$ are positive.
Thus, it enough to check for $D_+(T)= 1-3\beta$.
For $D_+(T) = D_+(e_0) = 1-3\beta$, \eqref{eqn:p=3D(e_0)f} becomes $$
D(e_0)f \ge (1-c-(2-c)\beta )(1-3\beta) >0.$$
Hence, we have proved the inequality in Lemma~\ref{lma:p=3sumk_3}.

It is easy to check that if equality holds in \eqref{eqn:p=3t=2sumk_tstrengthen} then~$D_+(T) = 0$.
Thus, for all edges~$e$ in~$T$, $D_+(e) =0$ by Lemma~\ref{lma:D(S)basicproperties}.
Furthermore, equality holds in \eqref{eqn:p=3t=2tildeD(T)}, so equality holds in Corollary~\ref{cor:sumk_tbasic} as $D_+(T)=0=D_+(e)$.
Hence, $d(v) = (1-\beta)n$ for~$v \in S$.
This completes the proof of the lemma.
\end{proof}

Together with Lemma~\ref{lma:sumwidetildeD(S)} with $t=2$, we obtain the strengthening of~\eqref{eqn:p=3t=2}.

\begin{cor} \label{cor:p=3k3k2k4}
Let~$1/4 \le \beta <1/3$.
Suppose~$G$ is a graph of order~$n$ with minimum degree~$(1-\beta) n$.
Then 
\begin{align*}
&(1+3\beta)k_3+\frac{2}{1-2\beta}\left( 1-3\beta + \frac{4\beta-1}{29-75\beta}\right) \sum_{T \in \mathcal{K}_3} D_+(T)  \ge 2(1-2\beta)\beta n k_2 +4\frac{k_4}n
\end{align*}
holds.
Moreover, if equality holds, then $G$ is~$(1-\beta)n$-regular and for each edge~$e$, either we have~$D(e) = 1-2\beta$ or~$D(e) = 2\beta$. 
\hfill{$\square$}
\end{cor}

Note that by mimicking the proof of Lemma~\ref{lma:p=3sumk_3}, we could obtain a strengthening of Corollary~\ref{cor:sumk_tbasicimproved} for~$t=3$.
It would lead to a strengthening of~\eqref{eqn:p=3t=3}.
However, it is still not sufficient to prove the Conjecture~\ref{conj} when $\beta$ is close to~$1/3$.
Instead, we prove the following statement.
The proof requires a detailed analysis of~$\mathcal{K}_5$, so it is postponed to Section~\ref{sec:p=3thmkpkp+1}.
\begin{lma} \label{lma:p=3t=3}
Let~$1/4 \le \beta <1/3$.
Suppose~$G$ is a graph order~$n$ with minimum degree~$(1-\beta) n$.
Then 
\begin{align}
(2-4\beta)k_4 & \ge (1-3\beta)\beta n k_3 +\left( 1-3\beta +\frac{4\beta-1}{29-75\beta} \right)n \sum_{T \in \mathcal{K}_3} D_+(T).
\end{align}
Moreover, equality holds only if~$(n,\beta)$ is feasible, and~$G\in \mathcal{G}(n,\beta)$.
\end{lma}

By using the two strengthened versions of \eqref{eqn:p=3t=2} and \eqref{eqn:p=3t=3}, that is, Corollary~\ref{cor:p=3k3k2k4} and Lemma~\ref{lma:p=3t=3}, we prove the theorem below, which implies Theorem~\ref{thm:p=3conj2}.

\begin{thm} \label{thm:p=3conj}
Let~$1/4 \le \beta <1/3$.
Let $s$ and~$t$ be integers with~$2\leq t < s \leq 4$.
Suppose~$G$ is a graph of order~$n$ with minimum degree~$(1-\beta) n$.
Then 
\begin{align*}
	\frac{k_s(G)}{g_s(\beta)n^s} & \ge \frac{k_t(G)}{g_t(\beta)n^t}.
\end{align*}
Moreover, the following three statements are equivalent:
\begin{itemize}
   \item[(i)] Equality holds for some~$2\leq t < s \leq 4$.
   \item[(ii)] Equality holds for all~$2\leq t < s \leq 4$.
   \item[(iii)] The pair~$(n,\beta)$ is feasible, and~$G$ is a member of~$\mathcal{G}(n, \beta)$.
\end{itemize}
\end{thm}

\begin{proof}
Recall that $p=3$ as $1/4 \le \beta <1/3$, so
\begin{align*}
 g_2(\beta) = (1-\beta)/2, \ g_3(\beta) =  (1-2\beta)^2\beta \textrm{ and } g_4 = (1-2\beta)(1-3\beta)\beta^2/2.
\end{align*}
Note that in proving the inequality, it is sufficient to prove the case when~$s =t+1$.
Lemma~\ref{lma:p=3t=3} states that $(2-4\beta) k_{4} \geq (1-3\beta)\beta n k_3$.
This implies~$k_{4}/g_{4}(\beta) n^{4} \ge k_3/g_3(\beta) n^3$ by~\eqref{eqn:property-g2} with~$t=3$.
Hence, the theorem is true for~$t=3$. 
For $t=2$, by substituting Corollary~\ref{cor:p=3k3k2k4} into Lemma~\ref{lma:p=3t=3}, we obtain
\begin{align*}
 (1+3\beta)k_3 +\frac{2}{1-2\beta}\left( 1-3\beta + \frac{4\beta-1}{29-75\beta}\right) \sum_{T \in \mathcal{K}_3} D_+(T)  \ge   2(1-2\beta)\beta n k_2\\
 + \frac4{(2-4\beta)n} \left( (1-3\beta)\beta n k_3 +\left( 1-3\beta+\frac{4\beta-1}{29-75\beta} \right)n \sum D_+(T)\right).
\end{align*}
Observe that the $\sum D_+(T)$ terms on both sides cancel.
Hence, after rearrangement, we have $(1-\beta) k_3   \ge 2(1-2\beta)^2\beta n k_2$.
Thus, $k_{3}/g_{3}(\beta) n^{4} \ge k_2/g_2(\beta) n^3$ as required.

This is clear that~$(iii)$ implies~$(i)$ and~$(ii)$ by the construction of~$\mathcal{G}(n, \beta)$ and the feasibility of~$(n,\beta)$.
Suppose~$(i)$ holds, so equality holds for some $2 \le t < s \le 4$.
It is easy to deduce that equality also holds for $s=4$ and $t=3$. 
By Lemma~\ref{lma:p=3t=3}, $(n,\beta)$ is feasible, and $G \in \mathcal{G}(n,\beta)$.
\end{proof}

\section{Proof of Lemma~\ref{lma:p=3t=3}} \label{sec:p=3thmkpkp+1}

In this section, $T$, $S$ and~$U$ always denote a $3$-clique, $4$-clique and $5$-clique respectively.
Before presenting the proof, we recall some basic facts about $T$, $S$ and $U$.
Observe that $D_-(S) = 0$ for~$S \in \mathcal{K}_{4}$, so~$D_+(S) = D(S)$.
Recall that $\widetilde{D}(S)= \sum_{T \in \mathcal{K}_3(S)} D_-(T) - (2-4\beta)$.
Let $T_1, \dots, T_{4}$ be triangles in~$S$ with $D(T_i) \leq D(T_{i+1})$ for~$1\leq i \leq 3$.
Since $D_-(T) \le \beta$, we have
\begin{align}
 \widetilde{D}(S) = \begin{cases}
	 2(4\beta-1)	& \text{if } k_3^+(S) =4,\\
	 4\beta-1 + (D(T_1)-(1-3\beta)) & \text{if } k_3^+(S) =3,\\
	 D(T_1)+D(T_2) -2(1-3\beta) & \text{if } k_3^+(S) =2,
	\end{cases} \label{eqn:p=3widetilde(S)expanded}
\end{align}
where~$k_3^+(S)$ is the number of heavy triangles in~$S$.
Also recall that $D(T)\ge 1-3\beta$ by Lemma~\ref{lma:D(S)basicproperties}~$(i)$.
We will often make reference to these formulae throughout this section.

\begin{proof}[Proof of Lemma~\ref{lma:p=3t=3}]
Define the function $\eta: \mathcal{K}_{4} \rightarrow \mathbb{R}$ to be \begin{align}
	\eta(S)& = \widetilde{D}(S)- \frac{4\beta-1}{29-75\beta}  \sum_{T \in \mathcal{K}_3(S)}\frac{D_+(T)}{D_+(T)+\beta} \nonumber
\end{align}
for~$ S \in \mathcal{K}_{4}$.
Recall that for a heavy triangle~$T$, $D(T) =D_+(T)+\beta$.
Thus, only heavy 3-cliques in~$S$ contribute to $\sum D_+(T)/(D_+(T) + \beta)$.
We now claim that it is enough to show that $\sum_{S \in \mathcal{K}_4} \eta(S) \geq 0$. 
If $\sum_{S \in \mathcal{K}_4} \eta(S) \geq 0$, then Lemma~\ref{lma:sumwidetildeD(S)} with $t=3$ implies that
\begin{align}
	0 \le & \sum_{S \in \mathcal{K}_{4}} \eta(S) 
	= \sum_{S \in \mathcal{K}_{4}} \widetilde{D}(S)- \frac{4\beta-1}{29-75\beta}  n \sum_{T \in \mathcal{K}_{3}}D_+(T) \nonumber \\
	\leq & (2-4\beta)k_{4} - (1-3\beta) \beta n k_3   - \left(1-3\beta + \frac{4\beta-1}{29-75\beta}\right) n\sum_{T \in \mathcal{K}_3} D_+(T) \nonumber \\& +2 \sum_{S \in \mathcal{K}_4} D_+(S) - 10 k_5/n, \nonumber
\end{align}
where the last inequality is due to Lemma~\ref{lma:sumwidetildeD(S)} with $t=3$.
Observe that $\sum_{S \in \mathcal{K}_4} D_+(S) = \sum_{S \in \mathcal{K}_4} D(S) = 5k_5/n$, so the terms with $\sum D_+(S)$ and $k_5/n$ cancel.
Rearranging the inequality, we obtain the inequality in Lemma~\ref{lma:p=3t=3}.

Suppose $\sum_{S \in \mathcal{K}_4} \eta(S) < 0$.
Then, there exists a $4$-clique~$S$ with~$\eta(S) < 0$.
Such a $4$-clique is called \emph{bad}, otherwise it is called \emph{good}. 
The sets of bad and good $4$-cliques are denoted by $\mathcal{K}_{4}^{bad}$ and $\mathcal{K}_{4}^{good}$ respectively.
In the next claim, we identify the structure of a bad $4$-clique.

\begin{clm} \label{clm:p=3badKp}
Let $S$ be a bad $4$-clique.
Let $$\Delta =  (1 - 3 \beta)(1 + \epsilon) \textrm{ and } \epsilon = (4\beta-1)/(150\beta^2-137\beta+30).$$
Then, the following hold
\begin{description}
 \item[$(i)$] $S$ contains exactly one heavy edge and two heavy triangles,
 \item[$(ii)$] $0< D(S) < \Delta$,
 \item[$(iii)$] $D(T)+D(T')<  2\Delta$, where $T$ and~$T'$ are the two non-heavy triangles in~$S$.
\end{description}
\end{clm}

\begin{proof}
Let $T_1, \dots, T_{4}$ be triangles in~$S$ with $D(T_i) \leq D(T_{i+1})$ for~$1\leq i \leq 3$.
We may assume that $D_+(T_{4}) >0$, otherwise $S$ is good by Corollary~\ref{cor:sumk_tbasicimproved} as~${\eta}(S)=\widetilde{D}(S) \ge 0$.
Hence, $S$ is also heavy by Lemma~\ref{lma:D(S)basicproperties}~$(iv)$.
We separate cases by the number of heavy triangles in~$S$.

First, suppose all triangles are heavy.
Hence, $\widetilde{D}(S) = 2(4\beta-1)$ by~\eqref{eqn:p=3widetilde(S)expanded}.
Clearly, $D_+(T_i) \le 1-\beta$ for~$1 \le i \le 4$, so
\begin{align*}
	\eta(S) \geq & 2(4\beta -1 )- \frac{4\beta-1}{29-75\beta} \sum_{T \in \mathcal{K}_{3}(S)} \frac{D_+(T)}{D_+(T)+\beta} \\
	 \geq  &2(4\beta-1)  \left(1-  \frac{2(1-\beta)}{29-75\beta}\right)
	= \frac{2(4\beta-1) (27-73\beta)}{29-75\beta} \geq 0 .
\end{align*}
This contradicts the assumption that~$S$ is bad.
Thus, not all triangles in~$S$ are heavy, so $0<D(S)\le \beta$ by Lemma~\ref{lma:D(S)basicproperties}~$(v)$.
Also, $D_+(T) \le D_+(S) = D(S) \le \beta$.

Suppose all but one triangles are heavy, so $\widetilde{D}(S) \ge 4\beta-1$ by~\eqref{eqn:p=3widetilde(S)expanded}.
Hence, 
\begin{align*}
	\eta(S) \ge & 4\beta -1 - \frac{4\beta-1}{29-75\beta} \sum_{T \in \mathcal{K}_{3}(S)} \frac{D_+(T)}{D_+(T)+\beta} 
	\\ \geq & (4\beta-1) \left(1 -  \frac{3}{29-75\beta}\frac{D(S)}{D(S)+\beta}\right)\\ 
	 \geq  & (4\beta-1) \left(1 -  \frac{3}{2(29-75\beta)}\right) 
	=  \frac{5(4\beta-1)(11-30\beta)}{2(29-75\beta)}  \ge 0,
\end{align*}
which is a contradiction.

Suppose there is only one heavy triangle, $T_4$, in~$S$.
Corollary~\ref{cor:sumk_tbasic} implies that $\widetilde{D}(S) + D_+(T_4) \geq  2D_+(S)=2D(S)$.
Note that $D_+(T_4) \le D_+(S) = D(S)$, so $\widetilde{D}(S) \geq  D(S)$.
Thus, 
\begin{align*}
	\eta(S)  \ge & D(S) -  \frac{4\beta-1}{29-75\beta} \frac{D_+(T_4)}{D_+(T_{4})+\beta} 
	\geq  D(S) -   \frac{4\beta-1}{29-75\beta} \frac{D(S)}{D(S) + \beta}\\
	= & \left( 1 -   \frac{4\beta-1}{(29-75\beta)(D(S)+\beta)} \right)D(S)
	\geq  \left( 1 -   \frac{4\beta-1}{(29-75\beta)\beta} \right)D(S) > 0.
\end{align*}

Hence, $S$ has exactly two heavy triangles, namely $T_3$ and~$T_4$.
If~$D(S) \ge \Delta$, then 
\begin{align*}
	\eta(S) = & D(T_1) +D(T_2) - 2(1-3 \beta) - \frac{4\beta-1}{29-75\beta} \left(\frac{D_+(T_3)}{D_+(T_3)+\beta} + \frac{D_+(T_4)}{D_+(T_{4})+\beta} \right)\\
	\geq & 2(D(S) -(1-3\beta)) -   \frac{2(4\beta-1)}{29-75\beta} \frac{D(S)}{D(S) + \beta} \\
	> & 2(1-3\beta)\epsilon -\frac{2(4\beta-1)}{29-75\beta} \frac{\Delta}{\Delta + \beta}  
	\\ \ge &  2(1-3\beta)\epsilon -\frac{2(4\beta-1)\Delta}{(29-75\beta)(1-2 \beta)} =0.
\end{align*}
Thus, $D(S) < \Delta$.
If $D(T_1) +D(T_2) \ge 2 \Delta$, then $\widetilde{D}(S)\ge 2(\Delta-(1-3\beta)) = 2(1-3\beta)\epsilon$ by~\eqref{eqn:p=3widetilde(S)expanded}.
Moreover, since $D_+(T_i)\le D(S) < \Delta$ for $i=3,4$,
\begin{align*}
	\eta(S) 
	> & 2(1-3\beta)\epsilon -\frac{2(4\beta-1)}{29-75\beta} \frac{\Delta}{\Delta + \beta}  
	 \ge  2(1-3\beta)\epsilon -\frac{2(4\beta-1)\Delta}{(29-75\beta)(1-2 \beta)} =0.
\end{align*}
Thus, $(iii)$ is true.

We have shown that $S$ contains exactly two heavy triangles.
Therefore, to prove $(i)$, it is sufficient to prove that $S$ contains exactly one heavy edge.
A triangle containing a heavy edge is heavy by Lemma~\ref{lma:D(S)basicproperties}~$(iv)$.
Since $S$ contains two heavy triangle,  there is at most one heavy edge in~$S$.
It is enough to show that if $S$ does not contain any heavy edge and~$D(S) < \Delta$, then $S$ is good, which is a contradiction.
Assume that $S$ contains no heavy edge.
Let $e_i = T_i \cap T_4$ be an edge of~$T_{4}$ for~$i=1,2,3$.
We claim that~$\widetilde{D}(S) \geq D_+(T_4)$.
By Corollary~\ref{cor:sumk_tbasic} taking~$S=T_4$ and $t=2$, we obtain
\begin{align*}
	D(e_1) + D(e_2) +D(e_3) \geq & 2 - 3\beta + D(T_{4}) \\
	D(e_1) + D(e_2) \geq & 2-4\beta + D_+(T_{4}).
\end{align*}
as $D(e_3) \leq 2 \beta$ and $D_-(T_4) = \beta$.
By Lemma~\ref{lma:D(S)basicproperties}~$(ii)$, we get
\begin{align*}
D(T_1) + D(T_2) \geq D(e_1) +D(e_2) -2 \beta \geq 2(1-3\beta) +  D_+(T_4).
\end{align*}
Hence,~$\widetilde{D}(S) \geq D_+(T_4)$ by~\eqref{eqn:p=3widetilde(S)expanded}.
Therefore,
\begin{align*}
	  \eta(S) \ge &  {D_+(T_4)} - \frac{4\beta-1}{29-75\beta}   \sum_{T \in \mathcal{K}_{4}(S)} \frac{D_+(T)}{D_+(T)+\beta} \\
	\ge &   \left( 1 -  \frac{2(4\beta-1)}{(29-75\beta)(D_+(T_4)+\beta)}  \right) {D_+(T_4)}\\
	\geq &  \left( 1 -  \frac{2(4\beta-1)}{(29-75\beta)\beta}  \right) {D_+(T_4)}>0
\end{align*}
and so $S$ is good, a contradiction.
This completes the proof of the claim.
\end{proof}

Since a bad $4$-clique~$S$ must be heavy, that is, $D(S) >0$, it is contained in some $5$-clique.
A $5$-clique is called \emph{bad} if it contains at least one bad~$4$-clique.
We denote~$\mathcal{K}_{5}^{bad}$ to be the set of bad~$5$-cliques.
Define~$\widetilde{\eta}(S)$ to be $\eta(S) / D(S)$ for~$S \in \mathcal{K}_{4}$ with $D(S)>0$. 
Clearly,
\begin{align}
 n \sum_{S \in \mathcal{K}_4} \eta(S) =& \sum_{U \in\mathcal{K}_5 } \sum_{S \in \mathcal{K}_4(U)} \widetilde{\eta}(S) + n \sum_{S \in \mathcal{K}_4: D(S) =0} \eta(S). \label{eqn:sumwidetildeeta(S)}
\end{align}
Recall that our aim is to show that $\sum_{S \in \mathcal{K}_4} \eta(S)\ge 0$.
Since $D(S)=0$ implies that $S$ is good, we have $\eta(S) \ge 0$.
Hence, it is enough to show that $\sum_{S \in \mathcal{K}_4(U)} \widetilde{\eta}(S) \ge 0$ for each bad 5-clique~$U$.

Now, we give a lower bound on $\widetilde{\eta}(S)$ for bad 4-cliques~$S$. 
By Claim~\ref{clm:p=3badKp},
\begin{align}
 \eta(S) \ge & - \frac{4 \beta-1}{29-75\beta} \sum_{T \in \mathcal{K}_3(S)}\frac{D_+(T)}{D_+(T)+\beta}
	\geq - \frac{2(4 \beta-1)}{29-75\beta}\frac{D(S)}{D(S)+\beta}. \nonumber
\intertext{Hence,}
 \widetilde{\eta}(S) \ge & - \frac{2(4 \beta-1)}{(29-75\beta)(D(S)+\beta)} > -  \frac{2(4 \beta-1)}{(29-75\beta)\beta}. \label{eqn:p=3etakey}
\end{align}
Next, we are going to bound $D(S)$ above for~$S \in \mathcal{K}_{4}(U) \backslash \mathcal{K}_{4}^{bad}$ and $U \in \mathcal{K}^{bad}_5$.
Let $S^b \in \mathcal{K}_{4}^{bad}(U)$.
Observe that $S \cap S^b$ is a 3-clique.
Then, by Lemma~\ref{lma:D(S)basicproperties} and Claim~\ref{clm:p=3badKp}, we have
\begin{align}
D(S) \leq D(S \cap S^b) = D_+(S \cap S^b)+\beta  \leq  D(S^b)+\beta <  \Delta +\beta \label{eqn:p=3D(S)top}. 
\end{align}
Recall that a bad 4-clique~$S$ contains a heavy edge by Claim~\ref{clm:p=3badKp} and hence so does a bad 5-clique~$U$.
We split $\mathcal{K}^{bad}_5$ into subcases depending on the number of heavy edges in~$U$.
The next claim studies the relationship between the number of heavy edges and bad 4-cliques in a bad 5-clique~$U$.

\begin{clm} \label{clm:p=3badandheavyedges}
Let $U \in \mathcal{K}^{bad}_{5}$ with $h\geq 2$ heavy edges and $b$ bad 4-cliques.
Then $b \leq 2h/(h-1) = 2 +2/(h-1)$.
Moreover, if there exist two heavy edges sharing a common vertex,~$b \leq 3$.
\end{clm}

\begin{proof}
Define $H$ to be the graph induced by the heavy edges in~$U$.
Write $u_S$ for the vertex in~$U$ not in~$S \in \mathcal{K}_{4}(U)$.
This defines a bijection between $V(U)$ and~$\mathcal{K}_{4}(U)$.
If $S$ is bad, $u_S$ is adjacent to all but one heavy edges by Claim~\ref{clm:p=3badKp}.
By summing the degrees of $H$, $2h = \sum_{S \in \mathcal{K}_{4}(U)} d(u_S)\geq b(h-1)$.
Thus,~$b \le 2h/(h-1)$.

If there exist two heavy edges sharing a common vertex in~$H$, then every bad~$4$-clique must miss one of the vertices of these two heavy edges.
Hence,~$b \leq 3$. 
\end{proof}

\begin{clm} \label{clm:p=3twoheavy}
Let $U \in \mathcal{K}_5^{bad}$ with two heavy edges.
Then $\sum_{S \in \mathcal{K}_{4}(U)} \widetilde{\eta}(S) >0.$
\end{clm}

\begin{proof}
Let $e$ and $e'$ be two heavy edges in $U$, and let $b$ be the number of bad 4-cliques in $U$.
We consider the cases whether $e$ and $e'$ are vertex disjoint or not separately.
First, assume that $e$ and $e'$ are vertex disjoint.
Notice that $\sum_{S \in \mathcal{K}_{4}^{bad}(U)} \widetilde{\eta}(S) > -  b\gamma$ by~\eqref{eqn:p=3etakey}, where $\gamma= 2(4 \beta-1)/(29-75\beta)\beta$ and $b \le 4$ by Claim~\ref{clm:p=3badandheavyedges}.
Also, there is exactly one heavy $4$-clique~$S$ containing both $e$ and~$e'$.
Therefore, it is sufficient to prove that $\eta(S) \ge b D(S) \gamma$.
Since $S$ contains two disjoint heavy edges, all triangles in~$S$ are heavy by Lemma~\ref{lma:D(S)basicproperties}~$(iv)$.
Thus, $\widetilde{D}(S) = 2(4\beta-1)$ by~\eqref{eqn:p=3widetilde(S)expanded}.
Observe that $T = S \cap S'$ is a triangle for $S' \in \mathcal{K}_{4}(U) \backslash S$.
Moreover, $D_+(T)  \le D_+(S')=D(S')$ by Lemma~\ref{lma:D(S)basicproperties}~$(iii)$.
Hence,
\begin{align*}
 \eta(S) \ge & 2(4 \beta-1) - \frac{4 \beta-1}{29-75\beta} \sum_{S' \in \mathcal{K}_4(U) \backslash S} \frac{D(S')}{D(S')+\beta}\\
	> & (4\beta-1) \left( 2- \frac{1}{29-75\beta} \left(\frac{b\Delta}{\Delta+\beta} +\frac{(4-b)(\Delta+\beta)}{\Delta+2\beta} \right) \right)
\end{align*}
by Claim~\ref{clm:p=3badKp}~$(ii)$ and~\eqref{eqn:p=3D(S)top}.
Therefore, $\eta(S) - bD(S) \gamma$ is at least
\begin{align*}
& (4\beta-1) \left( 2- \frac{1}{29-75\beta} \left(\frac{b\Delta}{\Delta+\beta} +\frac{(4-b)(\Delta+\beta)}{\Delta+2\beta} \right) \right) - b(\Delta +\beta) \gamma\\
 \ge & (4\beta-1) \left( 2- \frac{4\Delta}{(29-75\beta)(\Delta+\beta)} \right) - 4(\Delta +\beta) \gamma>0.
\end{align*}
Thus, if $U$ contains two vertex disjoint heavy edges, $\sum_{S \in \mathcal{K}_{4}(U)} \widetilde{\eta}(S) >0$.
Similar argument also holds for the case when $e$ and $e'$ share a common vertex.
\end{proof}

Recall that a bad 5-clique contains at least one heavy edge.
Thus, we are left with the case $U \in \mathcal{K}^{bad}_{5}$ containing exactly one heavy edge.

\begin{clm} \label{clm:p=3onebad}
Suppose $U\in \mathcal{K}_5^{bad}$ with exactly one heavy edge~$e$.
Then, $\sum_{S \in \mathcal{K}_{4}(U)} \widetilde{\eta}(S) > 0$.
\end{clm}

\begin{proof}
Let $u_1, \dots, u_{5}$ be the vertices of $U$ with $u_{4}u_{5}$ is the heavy edge.
Write $S_i$ and $\eta_i$ to be $U - u_i$ and $\eta(S_i)$ respectively for~$1 \le i \le 5$.
Similarly  write $T_{i,j}$ to be~$U- u_i-u_j$ for~$1 \leq i< j \leq 5$.
Recall that a bad $4$-clique contains a heavy edge by Claim~\ref{clm:p=3badKp}~$(i)$.
Hence, $S_i$ is a bad $4$-clique only if $i\le 3$.
Without loss of generality, $S_1, \dots, S_{b}$ are the bad $4$-cliques in $U$.

Since $S_3$ contains a heavy edge, it contains at least 2 heavy triangles by Lemma~\ref{lma:D(S)basicproperties}~$(iv)$.
If $S_3$ contains either three or four heavy triangles, then $S_3$ is not bad by Claim~\ref{clm:p=3badKp}~$(i)$.
By a similar argument as in the proof of Claim~\ref{clm:p=3twoheavy}, we can deduce that $\eta_3 \ge 2 \gamma D(S_3)$, where as before $\gamma= 2(4 \beta-1)/(29-75\beta)\beta$.
Therefore, $\sum_{S \in \mathcal{K}_{4}(U)} \widetilde{\eta}(S) > 0$ as~$b \le 2$.
Thus, we may assume that there are exactly two heavy triangles in~$S_i$ for~$1 \le i \le 3$.
By Lemma~\ref{lma:D(S)basicproperties}~$(v)$, $D(S_i)< \beta$ for~$1 \le i \le 3$. 
For $1 \le i \le b$, $$D(T_{i,4})+ D(T_{i,5})< 2 \Delta = 2 (1-3 \beta)(1+\epsilon)$$
by Claim~\ref{clm:p=3badKp}~$(iii)$.
For~$b< i \le 3$, $\widetilde{D}(S_i) = D(T_{i,4})+ D(T_{i,5}) -2(1-3\beta)$ by~\eqref{eqn:p=3widetilde(S)expanded}.
Thus,
\begin{align}
	D(T_{i,4})+ D(T_{i,5}) 	
	= & \eta_i  + 2 (1-3 \beta) +\frac{4\beta-1}{29-75\beta} \sum_{T \in \mathcal{K}_3(S_i)} \frac{D_+(T)}{D_+(T)+\beta}\nonumber \\
	\leq & \eta_i  + 2 (1-3 \beta) + \frac{\gamma \beta D(S_i)}{D(S_i)+ \beta} \nonumber \\
	\leq & \eta_i  + 2 (1-3 \beta) + \gamma \beta/2. \nonumber 
\end{align}
After applying Corollary~\ref{cor:sumk_tbasicimproved} to~$S_{4}$ and~$S_{5}$ taking~$t=3$, and adding the two inequalities together, we obtain
\begin{align}
	2(2-4\beta) \le & \sum_{1\le i \le 3} \left( D_-(T_{i,4}) +D_-(T_{i,5})\right) + 2D_-(T_{4,5})   \nonumber \\
	2(2-5\beta) \le & \sum_{1 \le i\le b} \left( D(T_{i,4}) +D(T_{i,5})\right)+ \sum_{b< i \le 3} \left( D(T_{i,4}) +D(T_{i,5})\right) \nonumber\\
	< & 2b (1-3 \beta)(1+\epsilon)+ \sum_{b<i\le 3 } \eta_i + (3-b) \left(2 (1-3 \beta) +\gamma \beta /2 \right)  \nonumber \\ 
	2 (4\beta - 1) < & 2b (1-3 \beta)\epsilon +  \sum_{b<i\le 3 } \eta_i  + (3-b)\gamma \beta /2 \label{eqn:p=3onebad}
\end{align}
If $b=3$, the above inequality becomes $2(4\beta - 1) < 6 (1-3 \beta)\epsilon <2(4\beta - 1)$, which is a contradiction.
Thus, $b\le2$.
Notice that $\eta_i >- D(S_i) \gamma> -\gamma$ for $1 \le i \le b$.
Hence, $\sum_{S \in \mathcal{K}_{4}^{bad}(U)} \widetilde{\eta}(S)>- b\gamma$.
Also, recall that $D(S_i) \leq \beta$ for~$1 \le i \le 3$.
It is enough to show that~$\sum_{b <i \le 3 } \eta_i   \geq b \gamma \beta$. 
Suppose the contrary, so~$\sum_{b <i \le 3 } \eta_i < b \gamma \beta$. Then, \eqref{eqn:p=3onebad} becomes
\begin{align*}
	2 (4\beta - 1)< &2 b (1-3 \beta)\epsilon +  (3+b)\gamma \beta /2  \le 4(1-3\beta)\epsilon +5\gamma \beta/2 <2(4\beta-1), 
\end{align*}
which is a contradiction.
The proof of the claim is complete.
\end{proof}

Hence, by Claim~\ref{clm:p=3twoheavy} and Claim~\ref{clm:p=3onebad}, \eqref{eqn:sumwidetildeeta(S)} becomes $\sum_{S \in \mathcal{K}_4} \eta(S) \ge 0$,
so the inequality in Lemma~\ref{lma:p=3t=3} holds.
Now suppose equality holds in Lemma~\ref{lma:p=3t=3}.
Claim~\ref{clm:p=3twoheavy} and Claim~\ref{clm:p=3onebad} imply that no $5$ clique is bad, so no $4$-clique is bad.
Furthermore, we must have $\eta(S)=0$ for all~$S \in \mathcal{K}_{4}$.
It can be checked that if the definition of a bad $4$-clique includes heavy 4-cliques~$S$ with~$\eta(S)=0$, then all arguments still hold.
Thus, we can deduce that $G$ is $K_5$-free.
Hence, $G$ is also $K_5$-free.
By Theorem~4 taking $s=4$ and $t=3$, we obtain that $(n,\beta)$ is feasible and~$G \in \mathcal{G}(n,\beta)$. 
\end{proof}

\section{Proof of Theorem~\ref{thm:partial}} \label{sec:partial}

Here, we prove Theorem~\ref{thm:partial}.
Since the proof of theorem uses similar arguments in the proof of Theorem~\ref{thm:conjforkp+2free} and Lemma~\ref{lma:p=3sumk_3}, we only give a sketch of the proof.

\begin{proof}[Sketch of Proof of Theorem~\ref{thm:partial}]
For $ 2 \le t \le p$ and $1/(p+1) \le  \beta < 1/p$, define
\begin{align*}
 A_{t}^p(\beta) = & (t-1)((p+1)\beta-1) C_{t}^p(\beta) \textrm{, and}\\
 B_{t}^p (\beta) = & ((p+1)\beta-1) C_{t}^p(\beta),
\intertext{where $C_{j}(\beta)$ satisfies the recurrence}
 C_{t}(\beta) +1 = &(p-t+1)\beta C_{t-1}(\beta)
\end{align*}
with the initial condition $C_{p}(\beta)=0$ for $1/(p+1) \le \beta < 1/p$.
Explicitly, $C_{p-j}^p(\beta) = \sum_{0\le i < j} i! \beta^{i-j}/j!$ for $0 \le j \le p-2$.
These functions will be used as coefficients in corresponding statements of Lemma~\ref{lma:p=3sumk_3} for~$2 \le t < p$.
Define the integer~$r(\beta)$ to be the smallest integer at least 2 such that for~$r \le t \le p$, $A_{t}^p(\beta)  <   1$ and $B_t^p(\beta) <  (p-t)\beta$.
Let $$\beta_p = \sup \{\beta_0 : r(\beta)=2 \textrm{ for all } 1/(p+1) \le \beta <\beta_0 \}$$
and $\epsilon_p = \beta_p - 1/(p+1)$.
Observe that $A_t(\beta)$, $B_t(\beta)$ and $C_t(\beta)$ are right continuous functions of~$\beta$.
Moreover, both $A_{t}(\beta)$ and~$B_t(\beta)$ tend to zero as~$\beta$ tends $1/(p+1)$ from above, so~$\beta_p > 1/(p+1)$ and $\epsilon_p>0$.
By mimicking the poof of Lemma~\ref{lma:p=3sumk_3}, we have
\begin{align*}
\widetilde{D}(S) & \geq A_{t+1}^p(\beta) D_+(S) -  B_t^p(\beta) \sum_{T \in \mathcal{K}_{t}(S)}  \frac{D_+(T)}{D(T)} 
\end{align*}
for $S \in \mathcal{K}_{t+1}$, $1/(p+1)\le \beta < \beta_p$ and $2\le t \le p$.
Then, following the arguments in the proof of Theorem~\ref{thm:conjforkp+2free}, we can deduce that
\begin{align*}
	\frac{k_s(G)}{g_s(\beta)n^s} & \geq \frac{k_t(G)}{g_t(\beta)n^t} + \frac{1-t\beta -B_t^p(\beta)}{(1-t\beta)(p-t+1)\beta g_t(\beta)n^t} \sum_{T \in \mathcal{K}_t }D_+(T) 
\end{align*}
for $2 \leq t < s \le p+1$ and $1/(p+1)< \beta \le \beta_p$.
Since $1-t\beta -B_t^p(\beta) \ge 0$, the proof of theorem is completed.
\end{proof}

Clearly, $\epsilon_p$ defined in the proof is not optimal.
Generalising the proof of Lemma~\ref{lma:p=3t=3} would lead to an improvement on~$\epsilon_p$.

\section{Counting $(p+1)$-cliques} \label{sec:r=p+1}

In this section, we are going to prove the below theorem, which implies Theorem~\ref{thm:conjr=p+1}.

\begin{thm} \label{thm:r=p+1}
Let $0< \beta<1$ and $p= \lceil \beta^{-1} \rceil -1$.
Suppose $G$ is a graph of order~$n$ with minimum degree $(1-\beta)n$.
Then, for any integer~$2 \leq t\leq p$,
\begin{align*}
	\frac{k_{p+1}(G)}{g_{p+1}(\beta) n^{p+1}} & \geq \frac{k_t(G)}{g_t(\beta) n^t}.
\end{align*}
Moreover, for~$t=2$, equality holds if and only if $(n,\beta)$ is feasible, and $G$ is a member of~$\mathcal{G}(n, \beta)$.
\end{thm}

For positive integers~$2 \le t\leq s\le p+1$, define the function $\phi_t^s : \mathcal{K}_{s} \rightarrow \mathbb{R}$ such that
\begin{align*}
	\phi_t^s(S) = \begin{cases}
	D_-(S)	& \textrm{if } t=s, \textrm{ and} \\
	\sum_{U \in K_{s-1}(S)} \phi_t^{s-1}(U) & \textrm{if } t < s
	\end{cases}
\end{align*}
for~$S \in \mathcal{K}_{s}$.
Observe that for~$G_0 \in \mathcal{G}(n,\beta)$ with $(n,\beta)$ feasible,
\begin{align*}
 \phi_t^s(S) = \begin{cases}
	(s-t)! (1-t\beta)	& \text{if~$|V(S) \cap V_0|=0,1$}\\
	 (1-t \beta) s!/t! + ((p+1)\beta -1) (s-2)!/(t-2)!	& \text{if~$|V(S) \cap V_0| =2$}
	\end{cases}
\end{align*}
for $s$-cliques $S$ in $G_0$.
Let $\Phi_t^s(S)= \min \{ \phi_t^s(S),  \varphi_t^s\}$ for~$S \in \mathcal{K}_{s}$ and $2 \le t \le s \le p+1$, where
\begin{align*}
\varphi_t^s=(1-t \beta) s!/t! + ((p+1)\beta -1) (s-2)!/(t-2)!. 
\end{align*}
to be the analogue of $D_-$ for~$\phi_t^s$.
The next lemma gives a lower bound on~$\Phi_t^s(S)$ for~$S \in \mathcal{K}_{s}$.

\begin{lma} \label{lma:phibasic}
Let $0< \beta<1$ and $p=\lceil \beta^{-1} \rceil -1$.
Let $G$ be a graph of order~$n$ with minimum degree~$(1-\beta) n$.
Then, $$\Phi_t^s(S)  \ge (1-t \beta) s!/t! + \big(D_-(S)-(1-s\beta) \big) (s-2)!/(t-2)!$$
for~$S \in \mathcal{K}_s$ and $2 \le t < s \le p+1$.
In particular, for $s=p+1$ and $t=p$
\begin{align}
	\sum_{S \in K_{p+1}} \Phi_t^{p+1}(S)  
	 \ge & \left( (1-t \beta) \frac{(p+1)!}{t!}  - (1-(p+1)\beta) \frac{(p-1)!}{(t-2)!} \right) k_{p+1}. \label{eqn:phibasicp+1}
\end{align}
\end{lma}

\begin{proof}
Fix~$\beta$ and~$t$ and we proceed by induction on~$s$.
The inequality holds for $s=t+1$ by Corollary~\ref{cor:sumk_tbasicimproved}. 
Suppose $s \geq t+2$ and that the lemma is true for~$t, \dots, s-1$.
Hence
\begin{align*}
	\phi_{t}^s(S) = & \sum_{T \in \mathcal{K}_{s-1}(S)} \phi_t^{s-1}(T) \ge \sum_{T \in \mathcal{K}_{s-1}(S)} \Phi_t^{s-1}(T)
	\intertext{and by the induction hypothesis,}
	\geq &\sum_{T \in \mathcal{K}_{s-1}(S)} \left( (1-t \beta) \frac{(s-1)!}{t!} + \big(D_-(T)-(1-(s-1)\beta) \big) \frac{(s-3)!}{(t-2)!} \right)\\
	= & (1-t \beta) \frac{s!}{t!} + \frac{(s-3)!}{(t-2)!}\left( \sum_{T \in \mathcal{K}_{s-1}(S)} D_-(T) - s(1-(s-1)\beta)\right)  \\
	\ge & (1-t \beta) s!/t! + \big(D_-(S)-(1-s\beta) \big) (s-2)!/(t-2)!,
\end{align*}
where the last inequality comes from Corollary~\ref{cor:sumk_tbasicimproved} with~$t=s-1$.
The right hand side is increasing in $D_-(S)$.
In addition, the right hand side equals to $\varphi_t^s$ only if $D_-(S) = (p-s+1)\beta$.
Thus, the proof of the lemma is complete.
\end{proof}

Now, we bound~$\sum_{S \in \mathcal{K}_{s}} \Phi_t^s(S)$ from above using Proposition~\ref{prp:keyprp} to obtain the next lemma.
The proof is essentially a straightforward application of Proposition~\ref{prp:keyprp} with an algebraic check.

\begin{lma} \label{lma:phits}
Let $0< \beta<1$ and $p=\lceil \beta^{-1} \rceil -1$.
Let $G$ be a graph of order~$n$ with minimum degree~$(1-\beta) n$.
Then, for $2 \leq t\le s\le p+1$
\begin{align*}
	\sum_{S \in \mathcal{K}_{s}} \Phi_t^s(S) \leq & \varphi_t^{s-1} s k_{s} 
	 + 2((p+1)\beta-1) \sum_{i=t+1}^{s-1} \left(  \frac{(i-3)!}{(t-2)!}  k_{i}n^{s-i} \prod_{j=i}^{s-1}(1-j\beta)  \right) \\
	& +  \left( (t+1) k_{t+1}
	 - (p-t+1) \beta k_t n \right) n^{s-t-1} \prod_{j=t}^{s-1}(1-j\beta).
\end{align*}
\end{lma}

\begin{proof}
Fix $\beta$ and $t$. We proceed by induction on~$s$.
Suppose $s=t+1$.
Note that $\Phi_t^{t+1}(S) \le \sum_{T \in \mathcal{K}_{t}(S)}  D_-(T)$.
By Proposition~\ref{prp:keyprp}, taking  $\mathcal{A} = \mathcal{K}_{t}$, $f=D_-$, $g=D$, $M =(p-t+1)\beta$ and $m=1-t\beta$,
\begin{align*}
 	 & \sum_{S \in \mathcal{K}_{t+1}} \Phi_t^{t+1}(S) \le  \sum_{S \in \mathcal{K}_{t+1}} \sum_{T \in \mathcal{K}_{t}(S)}  D_-(T) = n \sum_{T \in  \mathcal{K}_{t}} D(T) D_-(T)\\
	\leq & (p-t+1)\beta n \sum_{T \in  \mathcal{K}_{t}} D(T) + (1-t\beta)n  \sum_{T \in  \mathcal{K}_{t}} D_-(T) - (1-t\beta)(p-t+1)\beta n k_{t}\\
	\leq & (t+1)(1-(p-2t+1)\beta) k_{t+1}- (1-t\beta)(p-t+1)\beta n k_{t}.
\end{align*}
Hence, the lemma is true for~$s=t+1$.
Now assume that $s\geq t+2$ and the lemma is true up to~$s-1$.
By Proposition~\ref{prp:keyprp} taking  $\mathcal{A}= \mathcal{K}_t$, $f = \Phi_t^{s-1}$, $g =D$, $M=\varphi_t^{s-1}$ and $m=1-(s-1)\beta$, we have
\begin{align*}
	& \sum_{S \in \mathcal{K}_{s}} \Phi_t^s(S) = n \sum_{T \in \mathcal{K}_{s-1}} D(T)\Phi_t^{s-1}(T)\\
	\leq &  \varphi_t^{s-1} \sum_{T \in \mathcal{K}_{s-1}} n D(T) + (1-(s-1)\beta) n \sum_{T \in \mathcal{K}_{s-1}}\Phi_t^{s-1}(T) 	- \varphi_t^{s-1} (1-(s-1)\beta) n k_{s-1}\\
	=& \varphi_t^{s-1} s k_s + (1-(s-1)\beta) n \sum_{T \in \mathcal{K}_{s-1}}\Phi_t^{s-1}(T) - \varphi_t^{s-1} (1-(s-1)\beta) n k_{s-1}.
\end{align*}
Next, we apply induction hypothesis on $\sum \Phi_t^{s-1}(T)$.
Note that $$(s-1)\varphi_t^{s-2} - \varphi_t^{s-1} = 2 ((p+1)\beta-1) (s-4)!/(t-2)!.$$
After collecting the terms, we obtain the desire inequality.
\end{proof}

Now we are ready to prove Theorem~\ref{thm:r=p+1}.
The proof is very similar to the proof of Theorem~\ref{thm:conjforkp+2free}.

\begin{proof}[Proof of Theorem~\ref{thm:r=p+1}]
We fix $\beta$ and write $g_t$ to be~$g_t(\beta)$.
We proceed by induction on~$t$ from above.
The theorem is true for $t=p$ by Lemma~\ref{lma:phibasic} and Lemma~\ref{lma:phits}.
Hence, we may assume~$t < p$.
By Lemma~\ref{lma:phits}, 
\begin{align*}
	\sum \Phi_t^{p+1}(S) \le  &(p+1)\varphi_t^{p}  k_{p+1} 
	 + 2((p+1)\beta-1) \sum_{i=t+1}^{p} \left(  \frac{(i-3)!}{(t-2)!} k_{i} n^{p+1-i}  \prod_{j=i}^{p}(1-j\beta)  \right) \\
	& +  \left( (t+1) k_{t+1} - (p-t+1) \beta n  k_t \right) n^{p-t} \prod_{j=t}^{p}(1-j\beta),\\
\intertext{and by the induction hypothesis}
	\le & (p+1)\varphi_t^{p}  k_{p+1} 
	 + 2((p+1)\beta-1) \sum_{i=t+1}^{p}  \left(\frac{k_{p+1}g_i}{g_{p+1}}  \frac{(i-3)!}{(t-2)!} \prod_{j=i}^{p}(1-j\beta) \right) \\
	& +  \left( (t+1) \frac{k_{p+1}}{g_{p+1}}g_{t+1} - (p-t+1) \beta n  k_t \right) n^{p-t} \prod_{j=t}^{p}(1-j\beta).
\end{align*}
Substitute the above inequality into \eqref{eqn:phibasicp+1} and rearranging to obtain the desire inequality.

Now suppose that equality holds, so equality holds in~\eqref{eqn:phibasicp+1}.
Therefore, $D(S) = D_-(S) = 0$ for all~$S \in \mathcal{K}_{p+1}$.
Thus, $G$ is $K_{p+2}$-free. 
By Theorem~\ref{thm:conjforkp+2free} $(n,\beta)$ is feasible, and~$G \in \mathcal{G}(n,\beta)$.
This completes the proof of the theorem.
\end{proof}

\section*{Acknowledgements}

The author is greatly indebted to Andrew Thomason for his comments and his help in making the proof clearer.

\bibliographystyle{amsplain}

\end{document}